\documentclass[a4paper,12pt]{amsart}

\usepackage{amssymb,enumerate,psfrag,graphicx,amsfonts,amsrefs,amsthm,mathrsfs,amsmath,amscd,version,graphicx}
\usepackage{xcolor}

\newtheorem{lemma}{Lemma}
\newtheorem{example}{Example}
\newtheorem{proposition}{Proposition}
\newtheorem{theorem}{Theorem}
\newtheorem{corollary}{Corollary}
\newtheorem{remark}{Remark}
\newtheorem{definition}{Definition}

\begin{document}
\keywords{Gain graph, Balance, Adjacency matrix, Spectrum, Cover graph, Group representation, Fourier transform.}

\title{A group representation approach to balance of gain graphs}

\author[M. Cavaleri]{Matteo Cavaleri}
\address{Matteo Cavaleri, Universit\`{a} degli Studi Niccol\`{o} Cusano - Via Don Carlo Gnocchi, 3 00166 Roma, Italia}
\email{matteo.cavaleri@unicusano.it}

\author[D. D'Angeli]{Daniele D'Angeli}
\address{Daniele D'Angeli, Universit\`{a} degli Studi Niccol\`{o} Cusano - Via Don Carlo Gnocchi, 3 00166 Roma, Italia}
\email{daniele.dangeli@unicusano.it}

\author[A. Donno]{Alfredo Donno}
\address{Alfredo Donno, Universit\`{a} degli Studi Niccol\`{o} Cusano - Via Don Carlo Gnocchi, 3 00166 Roma, Italia}
\email{alfredo.donno@unicusano.it}

\begin{abstract}
We study the balance of $G$-gain graphs, where $G$ is an arbitrary group, by investigating their adjacency matrices and their spectra.
As a first step, we characterize switching equivalence and balance of gain graphs in terms of their adjacency matrices in $M_n(\mathbb C G)$. Then we introduce a represented adjacency matrix, associated with a gain graph and a group representation, by extending the theory of Fourier transforms from the group algebra $\mathbb C G$ to the algebra $M_n(\mathbb C G)$. We prove that a gain graph is balanced if and only if the spectrum of the represented adjacency matrix associated with any (or equivalently all) faithful unitary representation of $G$ coincides with the spectrum of the underlying graph, with multiplicity given by the degree of the representation.
We show that the complex adjacency matrix of unit gain graphs and the adjacency matrix of a cover graph are indeed particular cases of our construction. This enables us to recover some classical results and prove some new characterizations of balance in terms of spectrum, index or structure of these graphs.
\end{abstract}

\maketitle

\begin{center}
{\footnotesize{\bf Mathematics Subject Classification (2010)}: 05C22, 05C25, 05C50, 20C15, 43A32.}
\end{center}

\section{Introduction}
A \emph{$G$-gain graph} is a graph where an element of a group $G$, called gain, is assigned to each oriented edge, in such a way that
the inverse element is associated with the opposite orientation. Gain graphs can be regarded as a generalization of \emph{signed graphs}, where the group is $\{-1,1\}$, extensively studied even beyond the graph theory (see \cite{zasbib} for an extended and periodically updated bibliography).
Gain graphs can  be also considered as particular cases of \emph{biased graphs} \cite{zaslavsky1}, which are graphs where a special subset of circles is selected.
The notion of gain graph is also strictly related to that of \emph{voltage graph}: however, the theory of voltage graphs especially moves towards graph-coverings (see  \cite{gross}) whereas among the key notions concerning signed and gain graphs there are switching equivalence and balance (see, for instance, \cite{Harary,zaslavsky1}).
In this paper, we are interested in studying the balance property for gain graphs on an arbitrary group, via a spectral investigation which is realized by using Group representation theory and a Fourier transform approach.\\
\indent The balance of signed graphs can be characterized in several equivalent ways: the positiveness of circles, the switching equivalence with the \emph{underlying graph} (that is, the same graph endowed with a positive constant signature), or equivalently, the existence of a \emph{potential function} on the vertices. Analogous characterizations hold for gain graphs (see \cite{zaslavsky1}).
But, for signed graphs,  also spectral characterizations of balance are largely studied, mostly by using the \emph{signed Laplacian matrix} and the \emph{signed adjacency matrix}, whose entries are in $\{-1,0,1\}$ (e.g. \cite{francescob, zasmat}). A classical result in this sense states that \emph{a signed graph is balanced if and only if the adjacency matrix is cospectral with the adjacency matrix of the underlying graph} \cite{acharya}. Zaslavsky asked in \cite{zasbib} for a generalization of this result to complex unit gain graphs: we give it for every group admitting a \emph{finite dimensional faithful unitary representation} (see Theorem \ref{teo}).\\
\indent The first issue one is dealing with, when working with gain graphs on an arbitrary group $G$, is to define an appropriate adjacency matrix allowing to develop spectral computations. The most natural adjacency matrix of a gain graph on $G$ has in fact entries in $G\cup\{0\}$. When $G$ is a subgroup of the multiplicative group of the complex numbers, such a matrix takes values in $\mathbb{C}$ and the spectrum can be studied in a classical sense.
Moreover, if the gains are unit complex numbers (that is, $G$ is a subgroup of the complex unit group $\mathbb T$), this natural adjacency matrix is Hermitian and its spectrum is real. This case has been largely investigated in the last decades, and some spectral characterizations of the balance have been given (see \cite{reff1,adun,shahulgermina}). On the other hand, in voltage graph theory the solution of this issue is given by the fact that the gains are permutations and the natural adjacency matrices associated with a voltage graph are those of their covering graphs.\\
\indent At the best of our knowledge, the present paper is the first attempt to study switching equivalence and balance in gain graphs by investigation of their adjacency matrices and their spectra, for an arbitrary group $G$, which is not necessarily a subgroup of $\mathbb{C}$, or an Abelian group, or a finite group.\\
\indent As a first step, we consider in Section \ref{section4} an abstract adjacency matrix $A$ for a $G$-gain graph, whose entries belong to the group algebra $\mathbb{C}G$. These matrices have already been studied in \cite{dalfo} in the context of covers of voltage digraphs. The matrix set $M_n(\mathbb{C}G)$ inherits a product and a trace from the group algebra and is in fact an algebra itself. We prove in Theorem~\ref{swe} that two gain graphs are switching equivalent if and only if their adjacency matrices are conjugated in $M_n(\mathbb{C}G)$ by a diagonal matrix with diagonal entries in $G$ (the analogue of the \emph{switching matrix} for signed graphs). Moreover, we prove in Theorem~\ref{thtraccia} that a gain graph is balanced if and only if the trace of each power of $A$ equals the trace of the same  power of the adjacency matrix of the underlying graph. Even more, for balance it is enough to check the asymptotic equivalence of the  two real sequences given by the  traces of their increasing powers.\\
\indent In order to work with complex matrices, as a second step, inspired by the existing theory for covers of voltage graphs (\cite{dalfo, gross, mizuno }), we introduce in Section \ref{sectionbalancenovembre} the \emph{represented adjacency matrix} $A_\pi$. The latter can be described starting from the adjacency matrix $A\in M_n(\mathbb{C}G)$ of the considered gain graph, and replacing each gain entry $g \in G$ by the matrix $\pi(g)$ associated with a complex representation $\pi$ of the group $G$. In this way, we obtain a sort of adjacency matrix of the gain graph, whose entries instead belong to $\mathbb{C}$. Moreover, if $\pi$ is unitary, the matrix is Hermitian and, if $\pi$ is  faithful, the represented adjacency matrix contains all information about the gain graph. If this is not the case, we  can consider simultaneously a complete system of irreducible representations, in such a way that their kernels have trivial intersection.
The construction of $A_\pi$ from $A$  is actually an extension of the \emph{Fourier transform} from the group algebra $\mathbb C G$ to $M_n(\mathbb{C}G)$, and its properties are shown in Proposition \ref{productfou}. This enables us to prove, in Theorem~\ref{teo}, that a gain graph is balanced if and only if for any (or, equivalently, for all) faithful unitary representation $\pi$ the spectrum of the represented adjacency matrix $A_\pi$ consists of $\deg(\pi)$ copies of the adjacency spectrum of the underlying graph. Therefore, Theorem~\ref{teo} strongly generalizes the main result of \cite{acharya}.\\
\indent It is remarkable that our results apply also to non-Abelian and infinite groups, and they hold for any unitary representation $\pi$. By the way, in the classical cases, our construction is coherent with the existent theory and produces, on the one hand, alternative proofs of known results, as in the case of complex unit gain graphs (see Corollaries \ref{toro} and \ref{primoautovalore}), on the other hand new criteria for balance, as in the case of finite groups (see Corollary \ref{bfiniti} and Remark \ref{abelian} for the Abelian case). In particular, it turns out that, for a connected gain graph on the group of the $n$-th roots of the unity, and so also for a connected signed graph, the balance can be checked only by means of the first eigenvalue of its adjacency matrix. For a connected gain graph on any finite group $G$, it suffices to look at the first eigenvalue, and its multiplicity, of the represented adjacency matrix $A_\pi$, where $\pi$ is any faithful unitary representation; alternatively, one can refer to a complete system of irreducible representations of $G$.\\
\indent In Section \ref{cover}, we focus our attention on the represented adjacency matrix associated with the \emph{left regular representation} of $G$: it turns out that this matrix is nothing but the adjacency matrix of the \emph{cover graph} (see Definition~\ref{coverdef}). On the one hand, we apply the results from the previous sections in order to give new, at this level of generality, characterizations of balance via the spectrum, the index and structure of the cover graph (see Theorem~\ref{coverb}). On the other hand, we give one more description of the decomposition of the spectrum of the cover graph in terms of the irreducible representations of the group, result first proved in \cite{mizuno}. It is a remarkable fact that our methods  can be also applied to other issues: for instance, every time the adjacency matrix of a $\mathbb T$-gain graph is also a represented adjacency matrix of a smaller gain graph, a decomposition of the
representation induces  a decomposition of the spectrum (see Example \ref{exampleklein}). This opens a challenge regarding the possibility of  recognizing gain graphs that are ``gain-cover''.\\
\indent It is worth mentioning that our effort of formally defining a Fourier transform on $M_n(\mathbb C G)$ in Section~\ref{fou} is rewarded also by our proof of Proposition~\ref{prop5} and its applications in Section~\ref{circulant}, where we show, in a rigorous and self-contained way, the existence of a decomposition of the spectrum of $G$-block circulant matrices (decomposition already appearing in \cite{g-circ}).

\section{Preliminaries on representations of groups}
Let $G$ be a (finite or infinite) group, with neutral element $1_G$,  and let $V$ be a vector space of dimension $k$ over $\mathbb C$. A \textit{representation} of $G$ on $V$ is a group homomorphism
$
\pi\colon G\to GL_k(V),
$
where $GL_k(V)$ is the general linear group on $V$, i.e., the set of all bijective $\mathbb{\mathbb C}$-linear maps from $V$ to itself. Denote by $M_k(\mathbb C)$ the set of all square matrices of size $k$ with entries in $\mathbb{C}$, and by $GL_k(\mathbb{C})$ the group of all invertible matrices in $M_k(\mathbb{C})$. Then $GL_k(V)$ can be naturally identified with $GL_k(\mathbb{C})$. With a small abuse of notation we denote by $\pi$ the homomorphism that associates with an element $g\in G$ the matrix in $GL_k(\mathbb{C})$ corresponding to $\pi(g)$. The dimension of $V$ is called the \emph{degree} of $\pi$, and it will be denoted $\deg(\pi)$. Roughly speaking, a representation is a way to represent the elements of a group as matrices acting bijectively on a vector space. The representation theory is a well studied area of research (see, for instance, \cite{fulton}, for the general theory).\\
\indent A representation $\pi$ is \emph{unitary} if $\pi(g)\in U_k(\mathbb{C})$, for each $g\in G$, with $U_k(\mathbb{C}) = \{M\in GL_k(\mathbb{C}) : M^{-1}=M^\ast\}$, where $M^*$ is the Hermitian transpose of $M$. Two representations $\pi$ and $\pi'$ of degree $k$ of a group $G$ are said to be equivalent, or $\pi\sim\pi'$, if there exists a matrix $S\in GL_k(\mathbb{C})$ such that, for any $g\in G$, it holds $\pi'(g)=S^{-1}\pi(g)S$. This is an equivalence relation and it is known that, for a finite group $G$, each class contains a unitary representative: this is the reason why working with unitary representations is not restrictive in the finite case. \\
\indent We will denote by $I_k$  the identity matrix of size $k$, and by $O_{k_1,k_2}$ the zero matrix of size $k_1\times k_2$.
 The \textit{kernel} of a representation $\pi$ is $\ker(\pi)=\{g\in G:\, \pi(g)=I_k\}$, a representation is said to be \textit{faithful} if $\ker(\pi)=\{1_G\}$. Once we have a non-faithful representation, we can quotient $G$ by $\ker(\pi)$ in order to get a faithful representation of the group $G/\ker(\pi)$.\\
The \textit{character} $\chi_\pi$ of a representation $\pi$ is the map $\chi_{\pi}:G\to \mathbb{C}$ defined by $\chi_\pi(g)=Tr(\pi(g))$, that is, the trace of the matrix $\pi(g)$. Two representations have the same character if and only if they are equivalent.

Given two matrices $A\in M_{n_1}(\mathbb C)$ and $B \in M_{n_2}(\mathbb C)$, one can construct the  \emph{direct sum}
 $A \oplus B\in M_{n_1+n_2}(\mathbb C)$ and the  \emph{Kronecker product} $A\otimes B\in M_{n_1n_2}(\mathbb C)$ as
\begin{align*}
A \oplus B &=\left(
         \begin{array}{cc}
           A & O_{n_1,n_2} \\
           O_{n_2,n_1} & B \\
         \end{array}
       \right);\\
(A\otimes B)_{n_2(i-1)+r,n_2(j-1)+p}&=A_{i,j} B_{r,p}\qquad i,j=1,\ldots,n_1,\quad r,p=1,\ldots, n_2.
\end{align*}
Given two representations $\pi_1$ and $\pi_2$ of $G$,  one can construct \emph{the direct sum representation $\pi=\pi_1 \oplus \pi_2$ of $G$},  defined by $\pi(g):=\pi_1(g)\oplus \pi_2(g)$, for every $g\in G$. Moreover we will use the notation $\pi^{\oplus i}$ for the $i$-iterated direct sum of $\pi$ with itself.

A representation $\pi$ of $G$ on $V$ is said to be \textit{irreducible} if there is no proper invariant subspace $W\subset V$ under the action of $G$. In formulae, there is no $W$ such that $\pi(g)w\in W$ for any $w\in W$ and for any $g\in G$. It is well known that if $G$ has $m$ conjugacy classes, there exists a list of $m$ irreducible, pairwise  inequivalent, representations $\pi_0,\ldots,\pi_{m-1}$, that we call a \emph{complete system of irreducible representations of $G$}. For every representation $\pi$ of $G$ there exist $k_0,\ldots,k_{m-1} \in \mathbb N \cup \{0\}$ such that
\begin{equation}\label{deco}
\pi\sim \bigoplus_{i=0}^{m-1} \pi_i^{\oplus k_i}\quad \mbox{ or, equivalently, } \quad \chi_\pi=\sum_{i=0}^{m-1} k_i \chi_{\pi_i}.
\end{equation}
We say that $\pi$ \emph{contains} the irreducible representation $\pi_i$ if  $k_i\neq 0$. In this paper, every time we use the caption \emph{-$\pi$ irreducible-} under a sum, an union or an intersection, we mean that $\pi$ varies in a complete system of irreducible representations of $G$.

We will focus on some special representations of $G$. The first one is the trivial representation $\pi_0\colon G\to \mathbb C$, with $\pi_0(g)=1$ for every $g\in G$. It is unitary and irreducible and it is always contained in a complete system of unitary irreducible representations of $G$. \\
\indent When $G$ is finite, it naturally acts by left multiplication on itself. This action can be regarded as the action  on the vector space (which has in fact the structure of a group algebra) $\mathbb{C}G=\{\sum_{x\in G} c_x x: \ c_x\in \mathbb{C}\}$. This gives rise to the \emph{left regular representation $\lambda_G:G\to GL_{|G|}( \mathbb{C}G)$} which is faithful and has degree $|G|$. More precisely
$
\lambda_G(g)(\sum_{x\in G} c_x x)= \sum_{x\in G} c_{g^{-1}x} x.
$
Moreover $\lambda_G$ contains each irreducible representation $\pi_i$ of $G$ with multiplicity $\deg (\pi_i)$, and Eq. \eqref{deco} becomes:
\begin{equation}\label{regdec}
\lambda_G \sim \bigoplus_{i=0}^{m-1} \pi_i^{\oplus \deg(\pi_i)}\quad \mbox{ or, equivalently, } \quad \chi_{\lambda_G}=\sum_{i=0}^{m-1} \deg(\pi_i) \chi_{\pi_i}.
\end{equation}
When $G$ is  the symmetric group $Sym(n)$ on $n$ elements, the regular representation $\lambda_{Sym(n)}$ has degree $n!$. In this case, we have another faithful representation, generated by the permutational action $\varrho_{n}$ of $Sym(n)$ on a set of $n$ objects, called \emph{permutation representation}. This action can be easily studied by considering the action of $Sym(n)$ on the vector space $\mathbb{C}^n$ by permutating coordinates. If $\{{\bf e}_1, {\bf e}_2, \ldots, {\bf e}_n\}$ is the standard basis of $\mathbb{C}^n$, then the action of $\tau\in Sym(n)$ is defined as $\tau ({\bf e}_i) = {\bf e}_{\tau(i)}$ for each $i=1,\ldots, n$. The representation $\varrho_n$ decomposes into two irreducible representations: $\varrho_n=\pi_0\oplus \pi_S$, where $\pi_0$ is the trivial representation and $\pi_S$ is called the \emph{standard representation}, given by the action of $Sym(n)$ on the subspace $\{{\bf x}=(x_1,\ldots, x_n)\in \mathbb{C}^n : \sum_{i=1}^n x_i=0\}$. Notice that the permutation representation can be defined for any group $G$ acting on a set $X$ by embedding $G$ into $Sym(|X|)$.

We end this section by introducing some notation about spectrum. For a matrix $M\in M_n(\mathbb C)$, we denote by $\sigma(M) = \{\lambda \in \mathbb{C} : \det (M-\lambda I_k) =0\}$ the \emph{spectrum} of $M$. Notice that, as we are working in the complex field, such a spectrum always consists of $k$ (not necessarily distinct) values, called the\emph{ eigenvalues} of $M$. We denote by $\rho(M):=\max\{|\mu|:\mu\in \sigma(M)\}$ the \emph{spectral radius} of $M$. In particular, if $M$ is Hermitian, one has $\sigma(M) \subset \mathbb R$ and we denote by $\lambda_1(M):=\max\{\mu:\mu\in \sigma(M)\}$ the \emph{index} of $M$.

\section{Fourier Transforms}\label{fou}
Let $G$ be a group. Consider the vector space $C_{f}(G)=\{f\colon G\to \mathbb C: supp(f) \mbox{ is finite} \}$ of the finitely supported functions on $G$, where $supp(f)\subset G$ is the support of $f$, that is, $supp(f) = \{x\in G: f(x) \neq 0\}$. If we define the convolution product
$$
(f\ast h)(x):=\sum_{y\in G} f(xy^{-1})h(y),
$$
then $C_{f}(G)$ becomes an algebra, with involution $f^*(x):= \overline{f(x^{-1})}$, for each $x\in G$. Notice that $supp (f\ast h)\subset \bigcup_{y\in supp(h)}supp(f)y$, that is finite.

In addition, consider the group algebra $\mathbb C G$ of finite $\mathbb C$-linear combinations of elements of $G$. An element $f\in \mathbb C G$ can be expressed as $f=\sum_{x\in G} f_x x$,
where the set $\{x\in G: f_x\neq 0\}$ is finite. There is an involution $^*$ on $\mathbb CG$ such that $f^*:=\sum_{x\in G} \overline{f_x} x^{-1}$ and a product that is the linear extension of that  of $G$:
$$
\left(\sum_{x\in G} f_x x\right)\cdot \left(\sum_{y\in G} h_y y\right):= \sum_{x,y\in G} f_x h_y \, x y, \qquad \mbox{ for each } f,h\in \mathbb CG.
$$
It is well known that the map $$f\in  C_{f}(G) \mapsto \sum_{x\in G} f(x) x\in \mathbb C G $$ is an algebra isomorphism preserving the involutions, and then, for any $x\in G$,
we will use the notation $f(x)\in \mathbb C$ or $f_x\in \mathbb C$, equivalently.

\begin{definition}\label{deffou1}
Let $\pi\colon G\to GL_{k}(\mathbb C)$ be a representation of $G$. For $f\in C_{f}(G)$ the \emph{Fourier transform} of $f$ at $\pi$ is
\begin{eqnarray*}\label{defifourieralf}
\hat{f}(\pi):=\sum_{x\in G} f(x) \pi(x)\in M_{k}(\mathbb C).
\end{eqnarray*}
\end{definition}

\begin{example}\label{ex2}\rm
Let $G=Sym(3)$ and $\varrho_3$ be its permutation representation (see Table \ref{permrepS3}).      \small{
\begin{table}
\begin{tabular}{|c|c|c|c|c|c|}
\hline
$1_{Sym(3)}$ & $(12)$ & $(13)$ & $(23)$ & $(123)$ & $(132)$ \\
\hline
$\left(
   \begin{array}{ccc}
     1 & 0 & 0 \\
     0 & 1 & 0 \\
     0 & 0 & 1 \\
   \end{array}
 \right)$  & $\left(
   \begin{array}{ccc}
     0 & 1 & 0 \\
     1 & 0 & 0 \\
     0 & 0 & 1 \\
   \end{array}
 \right)$ & $\left(
   \begin{array}{ccc}
     0 & 0 & 1 \\
     0 & 1 & 0 \\
     1 & 0 & 0 \\
   \end{array}
 \right)$ & $\left(
   \begin{array}{ccc}
     1 & 0 & 0 \\
     0 & 0 & 1 \\
     0 & 1 & 0 \\
   \end{array}
    \right)$  & $\left(
   \begin{array}{ccc}
     0 & 0 & 1 \\
     1 & 0 & 0 \\
     0 & 1 & 0 \\
   \end{array}
 \right)$ & $\left(
   \begin{array}{ccc}
     0 & 1 & 0 \\
     0 & 0 & 1 \\
     1 & 0 & 0 \\
   \end{array}
 \right)$ \\
\hline
\end{tabular}
\smallskip
\caption{The permutation representation $\varrho_3$ of $Sym(3)$.} \label{permrepS3}
\end{table}}\normalsize
Let $f = 1_{Sym(3)}+i(12) +(1+i)(13) + 2i(23) -i(123) -(132)\in \mathbb{C}Sym(3)$. We have:
\begin{eqnarray*}
\widehat{f}(\varrho_3) &=& \varrho_3(1_{Sym(3)}) + i \varrho_3(12) + (1+i)\varrho_3(13) + 2i\varrho_3(23) - i\varrho_3(123) - \varrho_3(132)\\
&=& \left(
      \begin{array}{ccc}
        1+2i & i-1 & 1 \\
        0 & 2+i & 2i-1 \\
        i & i & 1+i \\
      \end{array}
    \right).
\end{eqnarray*}
\end{example}
The following facts are well known and easy to prove.
\begin{proposition}\label{prodotto}
For every $f,h\in C_{f}(G)$, for every  representation $\pi,\pi'$ of finite degree, one has:
\begin{itemize}
\item $\widehat{(f+h)}(\pi)=\hat{f}(\pi)+\hat{h}(\pi)$;
\item $\widehat{f\ast h}(\pi)=\hat{f}(\pi)\hat{h}(\pi)$;
\item if $\pi \sim \pi'$ then the matrices $\hat{f}(\pi)$ and $\hat{f}(\pi')$ are similar;
\item $\hat{f}(\pi\oplus\pi')=\hat{f}(\pi)\oplus \hat{f}(\pi')$.
\end{itemize}
Moreover, if $\pi$ is unitary, one has
$\widehat{f^*}(\pi)=\hat{f}(\pi)^*.$
\end{proposition}

We are going to extend the previous notions to a matrix setting. Similar constructions can be found in several different contexts, from group rings to operator rings \cite{Chu, dalfo, Hurley, g-circ}.

Consider the algebra $C_{f}(G, M_n(\mathbb C))$ of the finitely supported functions from $G$ to $M_n(\mathbb C)$ endowed with the convolution product $(F\ast H)(x):=\sum_{y\in G} F(xy^{-1})H(y)$ and the involution $F^*(x):=F(x^{-1})^*$, where the $^*$ on the right is the Hermitian transpose in $M_n(\mathbb C)$.
Moreover, consider the algebra $M_n(\mathbb C G)$ over $\mathbb C$. An element  $F \in M_n(\mathbb C G)$ is a square matrix of size $n$ whose entry $F_{i,j}$ is an element of $\mathbb C G$. The product is $(F H)_{i,j}=\sum_{k=1}^n F_{i,k} H_{k,j},$
where $F_{i,k} H_{k,j}$ is the product in the group algebra $\mathbb C G$, and the involution $^*$ in $M_n(\mathbb C G)$ is defined by $(F^*)_{i,j}=(F_{j,i})^*$, where the $^*$ on the right is the involution in $\mathbb C G$.
\begin{proposition}\label{isomorfismo}
The algebras  $C_{f}(G, M_n(\mathbb C))$ and $M_n(\mathbb C G)$ are canonically isomorphic.
\end{proposition}
\begin{proof}
Let us define the map $\varphi \colon C_{f}(G, M_n(\mathbb C)) \to M_n(\mathbb C G)$ such that
$$
\varphi(F)_{i,j}:=\sum_{x\in G} F(x)_{i,j} x\in \mathbb C G
$$
for each $F\in C_{f}(G, M_n(\mathbb C))$, so that the coefficient multiplying $x$ in $\varphi(F)_{i,j}$ is
\begin{equation}\label{ug}
\varphi(F)_{i,j}(x)=  F(x)_{i,j}.
\end{equation}
Notice that, for $H\in M_n(\mathbb C G)$, we have $H_{i,j}\in \mathbb C G$ and   $\varphi^{-1}(H)\in C_{f}(G, M_n(\mathbb C))$ is such that $(\varphi^{-1}(H)(x))_{i,j}= H_{i,j}(x)$, for each $x\in G$.
For $F,H\in C_{f}(G, M_n(\mathbb C))$, for every $x\in G$ and $i,j=1,\ldots,n$, by using Eq. \eqref{ug} we get
\begin{eqnarray*}
\varphi \left(F\ast H\right)_{i,j}(x)\!&=&\! \left((F\!\ast\! H)(x)\right)_{i,j}\!  = \! \sum_{y\in G} \left(F(xy^{-1})H(y)\right)_{i,j}\! =\!
   \sum_{y\in G} \sum_{k=1}^n F(xy^{-1})_{i,k} H(y)_{k,j};  \\
 \left({\varphi (F) \varphi(H)}\right)_{i,j}(x) &=&
  \sum _{k=1}^n \left({\varphi(F)}_{i,k} {\varphi(H)}_{k,j} \right) (x)=
 \sum _{k=1}^n \sum_{y\in G }{\varphi(F)}_{i,k}(xy^{-1}) {\varphi(H)}_{k,j}(y)\\&=& \sum_{y\in G} \sum_{k=1}^n F(xy^{-1})_{i,k} H(y)_{k,j},
\end{eqnarray*}
so that $\varphi(F\ast H)=\varphi(F)\varphi(H)$.
Moreover:
\begin{align*}
\varphi(F^*)_{i,j}(x)&=  F^*(x)_{i,j}=\left(F(x^{-1})^*\right)_{i,j}= \overline{F(x^{-1})_{j,i}}=\overline{\varphi(F)_{j,i}(x^{-1}) }= (\varphi(F)_{j,i})^*(x)\\&=(\varphi(F)^*)_{i,j}(x).
\end{align*}

\end{proof}

\begin{example}\label{ex3} \rm
Consider the function $F\in C_{f}(Sym(3), M_2(\mathbb{C}))$  defined as:
$$
F(1_{Sym(3)}) =\left(
   \begin{array}{cc}
     1 & 0 \\
    0 & i \\
    \end{array}
 \right); \ F((12)) =\left(
   \begin{array}{cc}
     0 & -i \\
    1 & i \\
    \end{array}
 \right); \ F((13))= \left(
   \begin{array}{cc}
     2i & 0 \\
    0 & 1 \\
    \end{array}
 \right);
$$
$$
F((23))= \left(
   \begin{array}{cc}
     0 & 0 \\
     0 & 0 \\
    \end{array}
 \right); \ F((123))= \left(
   \begin{array}{cc}
     -i & 0 \\
    2 & 3i \\
    \end{array}
     \right); \ F((132))= \left(
   \begin{array}{cc}
     i & -1 \\
    0 & 1 \\
    \end{array}
 \right).
$$
The corresponding element $\varphi(F)$  in $M_2(\mathbb{C}Sym(3))$ is
\begin{eqnarray*}
 \left(
                                                                                       \begin{array}{cc}
                                                                                         F_{11} & F_{12} \\
                                                                                         F_{21} & F_{22} \\
                                                                                       \end{array}
                                                                                     \right)
= \scriptsize\left(
      \begin{array}{c|c}
        1_{Sym(3)} + 2i(13) -i(123) +i(132) & -i(12) -(132) \\    \hline
        (12)   +2(123) & i\ 1_{Sym(3)} +i(12) +(13) +3i(123) +(132) \\
      \end{array}
    \right).   \normalsize
\end{eqnarray*}
\end{example}
From now on, we do not distinguish these two algebras; more precisely,  for every $F\in C_{f}(G,M_n(\mathbb C))$, or equivalently $F\in M_n(\mathbb C G)$,
we will speak about $F(x)$  in $M_n(\mathbb C)$ and about  $F_{i,j}\in \mathbb C G$, so that $F(x)_{i,j}=F_{i,j}(x)\in \mathbb C$.

We can define on $M_n(\mathbb C G)$ an analogue of the Fourier transform of Definition \ref{deffou1}.

\begin{definition}\label{fouriermatrici}
Let $F\in M_n(\mathbb CG)$ and let $\pi$ be a representation of $G$ of degree $k$. The \emph{ Fourier transform of $F$ at $\pi$} is $$\widehat{F}(\pi):=\sum_{x\in G} F(x)\otimes \pi(x)\in M_{nk}(\mathbb C).$$
\end{definition}

The next proposition shows that the matrix  $\widehat{F}(\pi)$ is the matrix obtained from $F\in  M_n(\mathbb C G)$ by replacement of each element $x\in G$ with the square block    $\pi(x)$ of size $k$, and each $0\in \mathbb C G$ with a zero block of size $k$.
\begin{proposition}\label{blocchi}
The matrix $\widehat{F}(\pi)$ is a block matrix of  $n\times n$ blocks, whose block $\widehat{F}(\pi)_{i,j}$ has size $k$ and is of type $\widehat{F}(\pi)_{i,j}=\widehat{F_{i,j}} (\pi)$ for $i,j=1,\ldots,n$, where $\widehat{F_{i,j}} (\pi)$ is the  Fourier transform of $F_{i,j}\in \mathbb C G$ as in Definition \ref{deffou1}.
\end{proposition}
\begin{proof}
For all $i,j=1,\ldots, n$ and $r,p=1,\ldots, k,$ we have:
\begin{eqnarray*}
\widehat{F}(\pi)_{k(i-1)+r,k(j-1)+p}\! \!&=&\!\! \!\left(\sum_{x\in G} F(x)\otimes \pi(x)\right)_{k(i-1)+r,k(j-1)+p}\\
\!&=&\! \sum_{x\in G} (F(x)\otimes \pi(x) )_{k(i-1)+r,k(j-1)+p}\! =\!
\sum_{x\in G} F(x)_{i,j} \pi(x)_{r,p}\!=\!
(\widehat{F_{i,j}}(\pi))_{r,p}.
\end{eqnarray*}
\end{proof}

\begin{example}\rm
Consider again the group $Sym(3)$, and its permutation representation $\varrho_3$ of Table \ref{permrepS3}. Let $F\in M_2(\mathbb C Sym(3))$ be as in Example \ref{ex3}.
Then we have:
\begin{eqnarray*}
\widehat{F}(\varrho_3) &=& \sum_{x\in Sym(3)}F(x) \otimes \varrho_3(x)
   =\left( \begin{array}{cc}
                                                                                         \widehat{F_{11}}(\varrho_3) & \widehat{F_{12}}(\varrho_3) \\
                                                                                         \widehat{F_{21}}(\varrho_3) & \widehat{F_{22}}(\varrho_3) \\
                                                                                       \end{array}
                                                                                     \right) \\
&=&\tiny  \left(
      \begin{array}{c|c}
        \varrho_3(1_{Sym(3)}) + 2i\varrho_3(13) -i\varrho_3(123) +i\varrho_3(132) & -i\varrho_3(12) -\varrho_3(132) \\    \hline
        \varrho_3(12)  +2\varrho_3(123) & \varrho_3(1_{Sym(3)}) +i\varrho_3(12) +\varrho_3(13)  +3i\varrho_3(123) +\varrho_3(132) \\
      \end{array}
    \right) \\   \normalsize
    &=& \left(
          \begin{array}{ccc|ccc}
            1 & i & i & 0 & -1-i & 0 \\
            -i & 1+2i & i & -i & 0 & -1 \\
            3i & -i & 1 & -1 & 0 & -i \\ \hline
            0 & 1 & 2 & i & 1+i & 1+3i \\
            3 & 0 & 0 & 4i & 1+i & 1 \\
            0 & 2 & 1 & 2 & 3i & 2i \\
          \end{array}
        \right).
\end{eqnarray*}
\end{example}
The following proposition is the matrix analogue of Proposition \ref{prodotto}.
\begin{proposition}\label{productfou}
For every $F,H\in M_n(\mathbb C G)$, for every  representation $\pi,\pi'$ of finite degree, we have
\begin{itemize}
\item $\widehat{(F+H)}(\pi)=\hat{F}(\pi)+\hat{H}(\pi)$;
\item $\widehat{F H}(\pi)=\hat{F}(\pi)\hat{H}(\pi)$;
\item if $\pi \sim \pi'$ then the matrices $\hat{F}(\pi)$ and $\hat{F}(\pi')$ are similar;
\item $\hat{F}(\pi\oplus\pi')$ is similar to $\hat{F}(\pi)\oplus \hat{F}(\pi')$.
\end{itemize}
Moreover, if $\pi$ is unitary, we have
$\widehat{F^*}(\pi)=\hat{F}(\pi)^*.$
\end{proposition}
\begin{proof}
By virtue of Proposition \ref{blocchi} and taking into account Proposition \ref{prodotto}, we have:
\begin{align*}
&\widehat{F+H}(\pi)_{i,j}=\widehat{(F_{i,j}+H_{i,j})}(\pi)=\widehat{F}(\pi)_{i,j}+\widehat{H}(\pi)_{i,j}=(\widehat{F}(\pi)+\widehat{H}(\pi))_{i,j};\\
&\widehat{FH}(\pi)_{i,j}=\widehat{(FH)_{i,j}}(\pi)= \sum_{k=1}^n \widehat{F_{i,k} H_{k,j}}(\pi)=  \sum_{k=1}^n \widehat{F_{i,k}}(\pi)\widehat{H_{k,j}}(\pi)=\left(\widehat{F}(\pi)\widehat{H}(\pi)\right)_{i,j}.\end{align*}
If $\pi$ is unitary we have  $\widehat{F^*}(\pi)_{i,j}= \widehat{(F_{i,j})^*}(\pi)=\widehat{F_{i,j}}(\pi)^*=  \widehat{F}(\pi)^*_{i,j}.$
\\ Suppose  $\pi \sim \pi'$ so that there exists an invertible matrix $S$ of size $k$ such that $S^{-1}\pi(x)S=\pi'(x)$ for all $x\in G$. It is easy to prove that
$(I_n\otimes S)^{-1} \hat{F}(\pi) (I_n\otimes S)=\hat{F}(\pi')$.
\\ Now let $\deg (\pi) = k$ and $\deg (\pi')=k'$. Any integer $m\in\{1,2,\ldots, n(k+k')\}$ can be written as $m=l(k+k')+q$, with $l\in\{0,\ldots, n-1\}$ and $q\in\{1,\ldots, k+k'\}$. Let us define $\gamma \in Sym(n(k+k'))$ such that
$$\gamma (l(k+k')+q)=\begin{cases}
lk+q \quad &\mbox{ if } q\leq k\\
(n-1)k+lk'+q \quad &\mbox{ if } q>k.
\end{cases}$$
 Let $P_\gamma \in M_{n(k+k')}(\mathbb C)$ be the  permutation matrix associated with $\gamma$. By an explicit computation one obtains
$$
P_\gamma^{-1} \hat{F}(\pi\oplus\pi') P_\gamma =\hat{F}(\pi)\oplus \hat{F}(\pi').
$$
\end{proof}

\subsection{$G$-block circulant matrices}\label{circulant}
As an application of Proposition \ref{productfou}, we can compute the spectrum $\sigma(\mathcal M)$ of a \emph{$G$-block circulant matrix $\mathcal M$} of size $nk$, for a finite group $G$ of order $k$.
Similar spectral computations have been performed in \cite{g-circ}, generalizing the results of \cite{tee} about cyclic-block circulant matrices, and of \cite{ziz} about $G$-circulant matrices. However, we devote the present section to this subject because of its relation with \emph{cover graphs} (which will be studied in Section \ref{cover} in the setting of gain graphs) and since it has a very natural interpretation in our context.  \\
Consider a block matrix $\mathcal M:=\left(
                    \begin{array}{c|c|c|c}
                       \mathcal M_{1,1} & \mathcal M_{1,2} & \cdots & \mathcal M_{1,k} \\\hline
                       \mathcal M_{2,1} & \mathcal M_{2,2} & \cdots & \mathcal M_{2,k} \\   \hline
                         \vdots & \vdots & \ddots  & \vdots \\ \hline
                       \mathcal M_{k,1} & \mathcal M_{k,2} & \cdots & \mathcal M_{k,k} \ \end{array}
                  \right) \in M_{nk}(\mathbb C)$, where $\mathcal M_{r,p}$ is a block of size $n$.

The matrix $\mathcal M$ is $G$-block circulant if, for a fixed order of $G=\{g_1,g_2,\ldots,g_k\}$, anytime $g_{r}g_{p}^{-1}=g_{r'} g_{p'}^{-1}$ in $G$, we have $\mathcal M_{r,p}=\mathcal M_{r',p'}$. Let us fix $g_1=1_G$, and notice that we have at most $k$ distinct blocks in $\mathcal M$, each one appearing once in every row and every column. Observe that, if $G$ is the cyclic group of order $k$, one obtains exactly the classical definition of a \emph{block circulant matrix} (see, for instance, \cite{tee}).

With a given $G$-block circulant matrix, one can associate an element $M\in M_n(\mathbb CG)$ (or equivalently, $M\in C_{f}(G,M_n(\mathbb C))$) such that
\begin{equation}\label{M}
M(g_r):= \mathcal M_{r,1}.
\end{equation}
Notice that, by virtue of the $G$-block circulant property of $\mathcal M$, we have $M(g_r g_p^{-1}) =\mathcal M_{r,p}\in M_n(\mathbb C)$ for every $r,p =1,\ldots, k$.

\begin{proposition}\label{prop5}
Let $G$ be a finite group of order $k$ and $\mathcal M$ be a $G$-block circulant matrix. Then there exists a permutation matrix $P$ such that $P^{-1} \mathcal M P=\hat{M}(\lambda_G)$, where $\hat{M}(\lambda_G)$ is the Fourier transform of  $M\in M_n(\mathbb CG)$, as defined in Eq. \eqref{M},  at the left regular representation of $G$. Moreover
$$
\sigma(\mathcal M)=\bigsqcup_{\pi \mbox{ \tiny irreducible} } \deg(\pi) \sigma(\hat{M}(\pi)).
$$
\end{proposition}
\begin{proof}
First observe that
$\lambda_G(g)_{r,p}=
\begin{cases}
1 \quad &\mbox{if } g=g_r g_p^{-1}\\
0 \quad &\mbox{otherwise}.\end{cases}$ This implies that, for all $i,j=1,\ldots, n$, and all $r,p=1,\ldots, k$, one gets that
$ M(g)_{i,j} \lambda_G(g)_{r,p}$ is non-zero only if $g=g_r g_p^{-1}$.
From this,
\begin{eqnarray*}
\hat{M}(\lambda_G)_{k(i-1)+r,k(j-1)+p}&=&\sum_{g\in G} (M(g)\otimes \lambda_G(g))_{k(i-1)+r,k(j-1)+p}=\sum_{g\in G} M(g)_{i,j} \lambda_G(g)_{r,p}\\
&=& M(g_r g_p^{-1})_{i,j}=(\mathcal M_{r,p})_{i,j}=\mathcal M_{n(r-1)+i,n(p-1)+j}.
\end{eqnarray*}
The matrix $P$ is the matrix associated with the permutation $p\in Sym(nk)$ such that $p(k(i-1)+r)=n(r-1)+i$, that is, a \emph{perfect shuffle $n\times k$}, whose conjugation action changes the order of the matrices in a  Kronecker product (see for example \cite{shu}).
The second statement follows from the decomposition of Eq. \eqref{regdec} and from Proposition \ref{productfou}.
\end{proof}
Notice that, if $G$ is Abelian, then all its irreducible representations have degree $1$ and indeed coincide with their characters. In this case, the spectrum of $\mathcal M$ is the disjoint union of the spectra of $k$ matrices of size $n$, each obtained by linear combination of the blocks of $\mathcal M$ with coefficients given by the irreducible characters:
$$
\sigma(\mathcal M)=\bigsqcup_{\chi \mbox{ \tiny irreducible} }\sigma\left( \sum_{r=1}^k\chi(g_r) \mathcal M_{r,1} \right).
$$
In particular, if $G$ is the cyclic group $G=\{1,\xi,\xi^2, \ldots, \xi^{k-1}: \ \xi= e^\frac{2\pi i}{k}\}$ given by the $k$-th roots of the unity, then by using the classification of the characters of cyclic groups, one obtains
$$
\sigma(\mathcal M)=\bigsqcup_{l=0}^{k-1} \sigma\left( \sum_{r=1}^k   \xi^{lr} \mathcal M_{r,1}\right),
$$
that is equivalent to a result of Tee \cite{tee}.

\section{Gain graphs and balance}\label{section4}
Let $\Gamma=(V_\Gamma,E_\Gamma)$ be a finite oriented simple graph such that $(v,u)\in E_\Gamma$  if and only if $(u,v)\in E_\Gamma$: if this is the case, we briefly write $u\sim v$ when the orientation is not relevant. Let $G$ be a group and consider a map $\Psi\colon E_\Gamma\to G$ such that $\Psi(u,v)=\Psi(v,u)^{-1}$. The pair $(\Gamma,\Psi)$ is a \emph{$G$-gain graph}  (or equivalently, a gain graph on $G$) and $\Psi$ is said to be a \emph{gain function}. Let us fix an order $v_1,v_2,\ldots, v_n$ in $V_\Gamma$. The adjacency matrix $A_{(\Gamma,\Psi)}\in M_n(\mathbb C G)$ is the matrix with entries
$$
(A_{(\Gamma,\Psi)})_{i,j}=
 \begin{cases}
 \Psi(v_i,v_j) &\mbox{if } (v_i,v_j)\in E_\Gamma
  \\  0 &\mbox{otherwise.}
\end{cases}
$$
Notice that, by virtue of the isomorphism presented in Section \ref{fou}, the matrix $A_{(\Gamma,\Psi)}$ can be regarded as an element of $C_{f}(G,M_n(\mathbb C))$, where, for every $g\in G$, the matrix $A_{(\Gamma,\Psi)}(g)\in M_n(\mathbb C)$ is the adjacency matrix of the (non-gain) oriented graph $\Gamma_g$ obtained by considering only the oriented edges whose gain equals $g$ in $(\Gamma,\Psi)$. Let us denote by $A^+\in M_n(\mathbb C)$ the adjacency matrix of the underlying graph $\Gamma$, which is a symmetric matrix with entries in $\{0,1\}$. The spectrum $\sigma(A^+)$ (resp. the spectral radius $\rho(A^+)$) will be referred as the spectrum $\sigma(\Gamma)$ (resp. the spectral radius $\rho(\Gamma)$) of the underlying graph $\Gamma$.

\begin{example}\label{esempiopaletta}\rm
In Fig. \ref{esempio1label} a gain graph $(\Gamma,\Psi)$ on $5$ vertices is represented, with $G=Sym(3)$. We use the convention that an oriented edge from a vertex $v_i$ to a vertex $v_j$ with label $\sigma\in Sym(3)$ is such that $\Psi(v_i,v_j)=\sigma$. We have:
$$
A_{(\Gamma,\Psi)} = \left(
                      \begin{array}{ccccc}
                        0 & (12) & 0 & 0 & 1_{Sym(3)} \\
                        (12) & 0 & (132) & (23) & 0 \\
                        0 & (123) & 0 & 1_{Sym(3)} & 0 \\
                        0 & (23) & 1_{Sym(3)} & 0 & (12) \\
                        1_{Sym(3)} & 0 & 0 & (12) & 0 \\
                      \end{array}
                    \right).
$$
\begin{figure}[h]
\begin{center}
\psfrag{v1}{$v_1$}\psfrag{v2}{$v_2$}\psfrag{v3}{$v_3$}\psfrag{v4}{$v_4$}
\psfrag{v5}{$v_5$}
\psfrag{(12)}{$(12)$}\psfrag{(123)}{$(123)$}\psfrag{id}{$1_{Sym(3)}$}\psfrag{(23)}{$(23)$}
\includegraphics[width=0.4\textwidth]{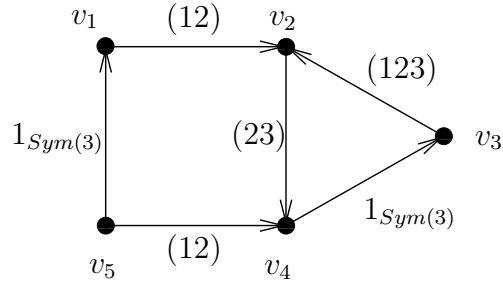}
\end{center}\caption{The gain graph $(\Gamma,\Psi)$ of Example \ref{esempiopaletta}.}  \label{esempio1label}
\end{figure}
\end{example}

\begin{proposition}\label{simmetry}
For a gain graph $(\Gamma,\Psi)$, the adjacency matrix $A_{(\Gamma,\Psi)}\in M_n(\mathbb C G)$ satisfies
$$
A_{(\Gamma,\Psi)}^*=A_{(\Gamma,\Psi)}.
$$
\end{proposition}
\begin{proof}
For $v_i,v_j\in V_\Gamma$ such that $(v_i,v_j)\in E_\Gamma$, we have
$$
(A_{(\Gamma,\Psi)}^*)_{i,j}=
((A_{(\Gamma,\Psi)})_{j,i})^*=
(\Psi((v_j,v_i)))^*=
\Psi((v_j,v_i))^{-1}=
\Psi(v_i,v_j)=
(A_{(\Gamma,\Psi)})_{i,j}.
$$
\end{proof}
Let $W$ be a \emph{walk} of length $h$, that is, an ordered sequence of $h+1$  vertices of $\Gamma$, say $v_0,v_1,\ldots, v_h$, with $v_i\sim v_{i+1}$. We can define
$$
\Psi(W):=\Psi(v_0,v_1)\cdots \Psi(v_{h-1},v_h).
$$
A \emph{closed walk} of length $h$ is a walk of length $h$ with $v_0=v_{h}$.

\begin{definition}
The gain graph $(\Gamma,\Psi)$ is balanced if $\Psi(C)=1_G$ for every closed walk $C$.
\end{definition}
\begin{remark}\label{osscicli} \rm
Actually it is not necessary to check the condition on every closed walk, but only on the simple ones (without repetitions of vertices, that are finitely many). Even more, it is enough to test the gain only on a \emph{fundamental system of circles}, see \cite[Corollary~3.2]{zaslavsky1} and \cite{rybzas2,rybzas1} for further developments in this direction.
\end{remark}

\begin{example}\label{esempio1quaternioni}\rm
Consider the gain graph in Fig. \ref{esempio2label}, where $G=Q_8 = \{\pm 1, \pm i, \pm j, \pm k\}$ is the quaternion group. The adjacency matrix is given by
$$
A_{(\Gamma,\Psi)} = \left(
                      \begin{array}{cccc}
                        0 & -k & i & k \\
                        k & 0 & j & 0 \\
                        -i & -j & 0 & j \\
                        -k & 0 & -j & 0 \\
                      \end{array}
                    \right).
$$
The graph $(\Gamma,\Psi)$ is balanced since every closed walk has gain $1$, as one can easily check. On the other hand, notice that the gain graph in Fig. \ref{esempio1label} is not balanced.
\begin{figure}[h]
\begin{center}
\psfrag{v1}{$v_1$}\psfrag{v2}{$v_2$}\psfrag{v3}{$v_3$}\psfrag{v4}{$v_4$}
\psfrag{k}{$-k$}\psfrag{-i}{$-i$}\psfrag{j}{$j$}
\includegraphics[width=0.35\textwidth]{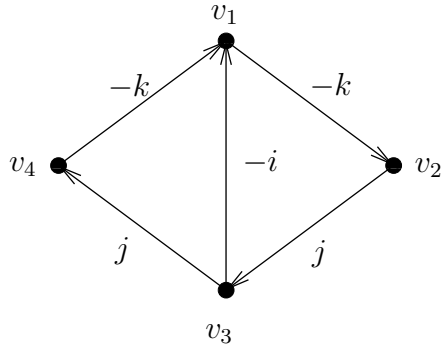}
\end{center}\caption{The gain graph of Example \ref{esempio1quaternioni}.}  \label{esempio2label}
\end{figure}
\end{example}
Another example of a balanced gain graph is the trivial one, where $\Psi(e)=1_G$ for any edge $e\in E_\Gamma$. We will denote by $\bold{1}_G$ the trivial gain function.
\\ A fundamental concept in the theory of gain graphs, inherited from the theory of signed graphs, is the \emph{switching equivalence}.
\begin{definition}\label{defswe}
Two gain functions $\Psi_1$ and $\Psi_2$ on the same underlying graph
$\Gamma$ are switching equivalent, and we shortly write $(\Gamma,\Psi_1)\sim(\Gamma,\Psi_2)$,  if there exists $f\colon V_\Gamma\to G$ such that
\begin{equation}\label{eqsw}
\Psi_2(v_i,v_j)=f(v_i)^{-1} \Psi_1(v_i,v_j)f(v_j), \qquad \forall (v_i,v_j)\in E_\Gamma.
\end{equation}

\end{definition}
It turns out that a gain graph $(\Gamma, \Psi)$ is balanced if and only if $(\Gamma, \Psi)\sim(\Gamma, \bold{1_G})$ (see \cite[Lemma~5.3]{zaslavsky1}). Moreover, in analogy with the signed case, the following result holds.
\begin{theorem}\label{swe}
Let $\Psi_1$ and $\Psi_2$ be two gain functions on the same underlying graph $\Gamma$, with adjacency matrices $A$ and $B\in M_n(\mathbb C G)$, respectively.
Then $(\Gamma,\Psi_1)\sim(\Gamma,\Psi_2)$ if and only if there exists a diagonal matrix $F\in M_n(\mathbb C G)$, with $F_{i,i}\in G$ for each $i=1\ldots, n$, such that $F^*AF=B$.
\end{theorem}
\begin{proof}
If $(\Gamma,\Psi_1)\sim(\Gamma,\Psi_2)$ there exists a map $f\colon V_\Gamma\to G$ such that Eq. \eqref{eqsw} holds.
The diagonal matrix $F\in M_n(\mathbb C G)$ with entries $F_{i,i}=f(v_i)$ satisfies $F^*AF=B$, since:
\begin{eqnarray*}
(F^*AF)_{i,j}&=&\sum_{r=1}^n\sum_{s=1}^n (F^*)_{i,r}{A}_{r,s} F_{s,j}=\sum_{r=1}^n\sum_{s=1}^n (F_{r,i})^*{A}_{r,s} F_{s,j}\\
&=& f(v_i)^{-1} {A}_{i,j} f(v_j)=
\begin{cases}
f(v_i)^{-1} \Psi_1(v_i,v_j)f(v_j) &\mbox{if } (v_i,v_j)\in E_\Gamma\\
0 &\mbox{otherwise,}
\end{cases}
\end{eqnarray*}
which corresponds exactly to the entry $B_{i,j}$ by Eq. \eqref{eqsw}.
Vice versa, if there exists a diagonal matrix $F\in M_n(\mathbb C G)$ such that $F^*AF=B$, we can define $f(v_i):=F_{i,i}$ and easily verify that Eq. \eqref{eqsw} holds.
\end{proof}

In what follows, we denote by $\mathcal W^h_{i,j}$ the set of walks  of length $h$ from $v_i$ to $v_j$.
We shortly write $A\in M_n(\mathbb C G)$ instead of $A_{(\Gamma,\Psi)}$, and we write $|A|\in M_n(\mathbb C G)$ instead of  $A_{(\Gamma,\bold{1}_G)}$.

\begin{lemma}\label{konig}
Let $A$ be the adjacency matrix of the $G$-gain graph $(\Gamma,\Psi)$.
We have
$$
(A^h)_{i,j}=\sum_{W\in\mathcal W^h_{i,j}} \Psi(W).
$$
\end{lemma}
\begin{proof}
We will prove the claim by induction on $h$. For $h=1$, the claim is true by definition of $\Psi$. Suppose that the equality holds for $h$. For fixed $i,j\in \{1,\ldots, n\}$ we have:
\begin{align*}
\sum_{\tiny W\in\mathcal W^{h+1}_{i,j}} \Psi(W)&= \sum_{v_{j'} \sim v_j} \sum_{\tiny W\in\mathcal W^h_{i,j'}}   \Psi(W) \Psi(v_{j'},v_j)= \sum_{j':v_j\sim v_{j'}} (A^h)_{i,j'} A_{j',j}\\&=\sum_{j'=1}^n (A^h)_{i,j'} A_{j',j}=(A^{h+1})_{i,j}.
\end{align*}
\end{proof}

Let us denote by $\mathcal C^h:=\bigcup_{i=1}^{n} \mathcal W_{i,i}^h$ the set of all closed walks of length $h$. We partition this set into balanced and unbalanced closed walks, so that ${\mathcal C}^h=\mathcal C_b^h\sqcup \mathcal C_u^h$,
with $\mathcal C_b^h:=\{C\in {\mathcal C}^h: \Psi(C)=1_G\}$ and  $\mathcal C_u^h:=\{C\in {\mathcal C}^h: \Psi(C)\neq1_G\}$.\\

A trace on the algebra $\mathbb C G$ can be defined as follows: for $f=\sum_{g\in G} f_g g$, we put $Tr(f)=f_{1_G}$ (or $f(1_G)$ regarding $f$ as an element of $C_{f}(G)$).
By composing with the usual trace on $M_n(\mathbb C)$, we define the trace of a matrix $A\in M_n(\mathbb C G)$ as:
\begin{equation}\label{trace}
Tr(A):=\sum_{i=1}^n Tr(A_{i,i}).
\end{equation}
Using Lemma \ref{konig} and Eq. \eqref{trace} we deduce
\begin{equation}\label{traccicli}
Tr(|A|^h)=|{\mathcal C}^{h}| \qquad\qquad Tr(A^h)=|{\mathcal C}_b^{h}|.
\end{equation}
Notice that also for the adjacency matrix $A^+\in M_n(\mathbb C)$ of the underlying graph $\Gamma$ we have $Tr((A^+)^h)=|{\mathcal C}^{h}|$ (in this case $Tr$ is the usual trace on $M_n(\mathbb C))$).
Let $\mu_1\geq\mu_2\geq\cdots\geq\mu_n$ be the eigenvalues of the underlying graph $\Gamma$; it is well known that $\rho(\Gamma)=\mu_1$ and, if
$\Gamma$ is connected, then $\mu_1>\mu_2$. Moreover $\mu_n=-\rho(\Gamma)$ if and only if $\Gamma$ is bipartite. Since $|\mathcal C^h|=Tr((A^+)^h)=\mu_1^h+\mu_2^h+\cdots+\mu_n^h$, for a connected graph $\Gamma$ we have, for $h\to\infty$
\begin{equation}\label{exp}
\begin{cases}
\rho(\Gamma)^h / |{\mathcal C}^h| \to 1\quad&\mbox{ if $\Gamma$ is not bipartite} \\
2\rho(\Gamma)^{2h}/ |{\mathcal C}^{2h}| \to 1\quad&\mbox{ if $\Gamma$ is  bipartite.}
\end{cases}
\end{equation}
(Notice that in a bipartite graph every closed walk has even length.)

\begin{theorem}\label{thtraccia}
Let $(\Gamma,\Psi)$ be a connected gain graph, $A\in M_n(\mathbb C G)$ its adjacency matrix and $|A|\in M_n(\mathbb C G)$  the adjacency matrix of $(\Gamma,\bold{1_G})$. The following are equivalent:
\begin{enumerate}
\item[(i)]
$(\Gamma,\Psi)$ is balanced;
\item[(ii)]
 $Tr(A^h)=Tr(|A|^h)$ for every $h\in \mathbb N$;
 \item[(iii)]
 $\begin{cases}
Tr(A^{h})/Tr(|A|^{h})=\frac{|C_b^h|}{|C^h|}\to1\quad&\mbox{ if $\Gamma$ is not bipartite}\\
Tr(A^{2h})/Tr(|A|^{2h})=\frac{|C_b^{2h}|}{|C^{2h}|}\to 1\quad&\mbox{ if $\Gamma$ is  bipartite}.
\end{cases}$
\end{enumerate}
\end{theorem}
\begin{proof}
(i)$\implies$(ii) easily follows by Eq. \eqref{traccicli}, and (ii)$\implies$(iii) is obvious.\\
(iii)$\implies$(i)\\
We are going to prove, by contradiction, that the negation of (i) implies the negation of (iii).
Suppose that there exists an unbalanced closed walk $W=v_0,\ldots,v_p$ with $v_i\in V_\Gamma$.
Without loss of generality, we can assume that $W$ visits all the vertices of $\Gamma$. Notice that, if $\Gamma$ is bipartite, $p$ must be even. Then each $v\in V_\Gamma$ is reached by $W$: let us denote by $i_v:=\min\{i: v_i=v\}$, and
by $W_v$ the closed walk $v=v_{i_v},v_{i_v+1},\ldots,v_p, v_0, v_1,\ldots,v_{i_v}=v$ (a \emph{shift} of the closed walk $W$, with different starting point).
We define a map $\varphi\colon {\mathcal C}_b^h\to {\mathcal C}^{h+p}_u$ mapping
$C\in {\mathcal C}_b^h$ to the concatenation of $C$ with $W_v$, where $v$ is the starting vertex of $C$.
It is clear that $\varphi(C)$ is in fact an unbalanced closed walk. Moreover, the map $\varphi$ is injective. Therefore for every $h\in \mathbb N$ such that $|{\mathcal C}^{h+p}|\neq 0$ and $|{\mathcal C}^{h}|\neq0$ we have
\begin{equation}\label{contrad}
\frac{|{\mathcal C}_u^{h+p}|}{|{\mathcal C}^{h+p}|}\geq \frac{|{\mathcal C}_b^{h}|}{|{\mathcal C}^{h+p}|}= \frac{|{\mathcal C}_b^{h}|}{|{\mathcal C}^{h}|}\cdot \frac{|{\mathcal C}^{h}|}{|{\mathcal C}^{h+p}|}.
\end{equation}
If the graph is not bipartite, by Eq. \eqref{exp}, we have that $|{\mathcal C}^{h}|/|{\mathcal C}^{h+p}|\to \rho(\Gamma)^{-p}> 0$.
Suppose, by the absurd, that (iii) holds, so that $\frac{|C_b^h|}{|C^h|}\to1$. Then $|{\mathcal C}_u^{h}|/|{\mathcal C}^{h}|$ and its subsequence $|{\mathcal C}_u^{h+p}|/|{\mathcal C}^{h+p}|$ tend to zero, that contradicts Eq. \eqref{contrad}. The proof in the case of a bipartite graph $\Gamma$ is analogous.
\end{proof}

\begin{remark}  \rm
We believe that the result of Theorem \ref{thtraccia} is interesting in itself; however, if $A$ and $|A|$ were complex matrices, this result would imply that the balance is equivalent to cospectrality between  $A$ and $|A|$ (see Lemma \ref{traccia} in Section \ref{sectionbalancenovembre}). At this purpose, in the next section, we introduce the \emph{Fourier transform of the adjacency matrices},  so that we can work with complex matrices, obtaining the wanted equivalence for a general gain graph.
\end{remark}

\section{Balance via spectra of represented adjacency matrices}\label{sectionbalancenovembre}
Let $\pi\colon G \to U_k(\mathbb C)$ be a unitary representation of $G$; let us denote by $\chi_\pi\colon G\to \mathbb C$ the associated character. Notice that $\chi_\pi(1_G)= \deg (\pi) = k$. In what follows, we denote by $\Re (z)$ the real part of the complex number $z$.

\begin{lemma}\label{reale}
For a group $G$ and a unitary representation $\pi\colon G \to U_k(\mathbb C)$, we have
\begin{equation}\label{ub}
\Re(\chi_\pi(g))\leq k
\end{equation} and the equality holds if and only if $g\in \ker (\pi)$.
\end{lemma}
\begin{proof}
For any $g\in G$ the value $\chi_\pi(g)$ is the sum of the (not necessarily distinct) eigenvalues $\lambda_1,\ldots, \lambda_k$ of $\pi(g)$. Since $\pi(g)$ is a unitary matrix, the spectrum lies in the unit circle, so that we have
$$\Re(\chi_\pi(g))\leq |\chi_\pi(g)|\leq |\lambda_1|+\cdots+|\lambda_k|= k,$$
which proves Eq. \eqref{ub}. Moreover, the equality  $k=\Re(\chi_\pi(g))=\Re(\lambda_1)+\cdots+\Re(\lambda_k)$ holds if and only if $\lambda_j=1$ for every $j=1,\ldots, k$, that is, if and only if $\pi(g)=I_k$.
\end{proof}

\begin{definition}
Let $\pi$ be a unitary representation of $G$ of finite degree, with character $\chi_\pi$. We define:
\begin{itemize}
\item the map ${Tr_\pi\colon \mathbb C G\to \mathbb C}$ such that $Tr_\pi(f):=Tr(\widehat{f}(\pi))$ for each $f\in \mathbb C G$;
\item the map $Tr_\pi\colon M_n(\mathbb C G)\to\mathbb C$ such that $Tr_\pi(F):=Tr(\widehat{F}(\pi))$ for each $F\in  M_n(\mathbb C G)$.
\end{itemize}
\end{definition}

\begin{remark}\label{remarknovembrino}\rm
Notice that, by Definition \ref{deffou1}, for $f\in \mathbb C G$ we have $Tr_\pi(f)=\sum_{g\in G} f(g) \chi_\pi(g)$.  Moreover, by Proposition \ref{blocchi}, for $F\in  M_n(\mathbb C G)$, we have
$Tr_\pi(F)=\sum_{i=1}^n Tr_\pi(F_{i,i})=\sum_{i=1}^n \sum_{g\in G} F_{i,i}(g) \chi_\pi(g).$
\end{remark}

\begin{example}\rm
Let  $\lambda_G$ be the left regular representation of $G$, so that $\deg(\pi)=|G|$. Then, by Eq. \eqref{trace} and Remark \ref{remarknovembrino}, we get, for each
$A\in M_n(\mathbb C G)$:
\begin{align*}
Tr_{\lambda_G}(A)=\sum_{i=1}^n Tr_{\lambda_G}(A_{i,i})= \sum_{i=1}^n \sum_{g\in G} A_{i,i}(g) \chi_{\lambda_G}(g)=  \sum_{i=1}^n |G| A_{i,i} (1_G)=
|G|Tr(A).
\end{align*}
\end{example}

For a $G$-gain graph $(\Gamma,\Psi)$ and a representation $\pi$ of $G$ of degree $k$, we introduce the \emph{represented adjacency matrix}
$$
A_{(\Gamma,\Psi,\pi)}:=\widehat{A_{(\Gamma,\Psi)}}(\pi)\in M_{nk}(\mathbb C).
$$
Roughly speaking, $A_{(\Gamma,\Psi,\pi)}$ is the matrix obtained from the matrix  $A_{(\Gamma,\Psi)}\in M_{n}(\mathbb C G)$ by replacing each occurrence of $g\in G$ with the block $\pi(g)$ and each $0\in \mathbb C G$ with a zero block of size $k$.
We shortly write $A_\pi$ instead of $A_{(\Gamma,\Psi,\pi)}$ and $|A|_\pi$ instead of $A_{(\Gamma,\bold{1}_G,\pi)}$. Notice that, when one considers the trivial representation $\pi_0$, one has $A_{\pi_0}=|A|_{\pi_0}=A^+$, which is the adjacency matrix of the underlying graph $\Gamma$.

\begin{proposition}\label{her}
If $\pi$ is a unitary representation of $G$, then the matrix $A_\pi$ is Hermitian.
\end{proposition}
\begin{proof}
It is a consequence of Propositions \ref{productfou} and \ref{simmetry}.
\end{proof}

\begin{proposition}\label{ll}
Let $\pi$ be a  faithful unitary representation of $G$ of finite degree. Then the graph $(\Gamma,\Psi)$ is balanced if and only if $Tr\left((A_\pi)^h\right)=Tr\left((|A|_\pi)^h\right)$ for all $h\in \mathbb N$.
\end{proposition}
\begin{proof}
Suppose $\deg(\pi)=k$. Observe that, by virtue of Proposition \ref{her}, the value $Tr\left((A_\pi)^h\right)$ is real.
By  Proposition \ref{productfou}, for every $F\in M_n(\mathbb C G)$, we have  $Tr((\widehat{F}(\pi))^h)=Tr(\widehat{F^h}(\pi))=Tr_\pi(F^h)$. It follows:
\begin{equation}\label{miserve}
\begin{split}
Tr\left((A_\pi)^h\right)&=Tr_\pi(A^h)=\sum_{i=1}^n Tr_\pi((A^h)_{i,i})= \sum_{i=1}^n Tr_\pi\left( \sum_{W\in\mathcal W_{i,i}^h}  \Psi(W)\right)= \sum_{W\in\mathcal C^{h}} \chi_\pi (\Psi(W))\\
&=
\Re\left( \sum_{W\in\mathcal C^{h}}  \chi_\pi (\Psi(W))\right)
=  \sum_{W\in\mathcal C^{h}} \Re(\chi_\pi (\Psi(W)))= \sum_{W\in\mathcal C_b^h} k +\sum_{W\in\mathcal C_u^{h}} \Re(\chi_\pi (\Psi(W)))
\\ &\leq \sum_{W\in\mathcal C^h} k= \sum_{i=1}^n \sum_{W\in\mathcal W_{i,i}^h} \chi_\pi (1_G)
= \sum_{i=1}^n Tr_\pi((|A|^h)_{i,i})=Tr\left((|A|_\pi)^h\right).
\end{split}
\end{equation}
Since $\pi$ is unitary and faithful, by  Lemma \ref{reale} the equality holds if and only if $\Psi(W)=1_G$ for all $W\in \mathcal C^h$ and $h\in\mathbb N$, which is the balance of $(\Gamma,\Psi)$.
\end{proof}

The following lemma is a consequence of Specht's theorem in the case of Hermitian matrices \cite{specht}.
\begin{lemma}\label{traccia}
For Hermitian matrices $A$ and $B$, we have $\sigma(A)=\sigma(B) \iff Tr(A^h)=Tr(B^h),\quad \forall h\in \mathbb N.$
\end{lemma}
Alternatively, elementary symmetric polynomials can be used to prove that, for $A,B\in M_n(\mathbb C)$, one has that $Tr(A^h)=Tr(B^h)$ for $h=1,\ldots,n$ is equivalent to $\sigma(A)=\sigma(B)$, see \cite{rob}.

The following is one of the main results of the present paper, connecting trace, spectral radius and spectrum of a represented adjacency matrix of $(\Gamma,\Psi)$ to its balance.

\begin{theorem}\label{teo}
Let $G$ be a group and let $\pi\colon G\to U_k(\mathbb C)$ be a faithful unitary representation of $G$. Let $(\Gamma,\Psi)$ be a  $G$-gain graph and denote by $A^+$ the adjacency matrix of its underlying graph. The following are equivalent:
\begin{enumerate}
\item[(i)] $(\Gamma,\Psi)$ is balanced;
\item[(ii)]   the spectrum $\sigma(A_\pi)$ consists  of $k$ copies of the spectrum $\sigma(A^+).$
\end{enumerate}
Moreover, if $G$ is finite and $\Gamma$ is connected, the following are equivalent to (i) and (ii).
\begin{enumerate}
\item[(iii)] The spectral radius $\rho(\Gamma)$ occurs in $\sigma(A_\pi)$ with multiplicity $k$;
\item[(iv)] $\begin{cases}
Tr(A_\pi^{h})/Tr(|A|_\pi^{h})\to1\quad&\mbox{ if $\Gamma$ is not bipartite}\\
Tr(A_\pi^{2h})/Tr(|A|_\pi^{2h})\to 1\quad&\mbox{ if $\Gamma$ is  bipartite}.
\end{cases}$
\end{enumerate}
\end{theorem}
\begin{proof}
${}$\\
(i)$\iff$(ii)\\
By Proposition \ref{ll}, the gain graph $(\Gamma,\Psi)$ is balanced if and only if $Tr\left((A_\pi)^h\right)\!=\!Tr\left((|A|_\pi)^h\right)$ for all $h\in \mathbb N$. By  Lemma \ref{traccia}, this is equivalent to $\sigma(A_\pi)=\sigma(|A|_\pi)$. Finally, being $|A|_\pi= A^+\otimes I_k$, the spectrum $\sigma(|A|_\pi)$ consists of $k$ copies of $\sigma(A^+)$ and the claim follows.\\
(ii)$\implies$(iii)\\
It is obvious. \\
(iii)$\implies$(iv)\\
By virtue of Eq. \eqref{miserve} we have
\begin{equation}\label{contra}
Tr\left((A_\pi)^h\right)\leq Tr\left((|A|_\pi)^h\right),\qquad \forall h\in \mathbb N.
\end{equation}
Let $\nu_1>\nu_2>\cdots>\nu_p$ be the eigenvalues of $A_\pi$ with multiplicity $m_1,\ldots, m_p$, respectively.
Since  $|A|_\pi= A^+\otimes I_k$ we have that  $\rho(|A|_\pi)=\rho(A^+)=\rho(\Gamma)$.
Since $\Gamma$ is connected, we have that $\rho(\Gamma)$ is an eigenvalue of $\Gamma$ of multiplicity $1$, and so it is an eigenvalue for $|A|_\pi$ of multiplicity $k$.
Moreover, the graph $\Gamma$ is bipartite if and only if also  $-\rho(\Gamma)$ is eigenvalue for $|A|_\pi$ with multiplicity $k$. Clearly
\begin{equation}\label{tracciaproof}
Tr\left((A_\pi)^h\right)=\sum_{i=1}^p m_i \nu_i^h,\qquad\qquad
Tr\left((|A|_\pi)^h\right)= k Tr\left((A^+)^h\right)= k |\mathcal C^h|.
\end{equation}
By Eq. \eqref{contra} we have:
\begin{equation}\label{spectralradius}
\rho(A_\pi)=\max\{|\nu_1|,|\nu_p|\}=\lim_{h\to\infty} \sqrt[2h]{Tr((A_\pi)^{2h})}\leq \lim_{h\to\infty} \sqrt[2h]{(Tr(|A|_\pi)^{2h})}=\rho(\Gamma).
\end{equation}
Since by hypothesis $\rho(\Gamma)$ is an eigenvalue of $A_\pi$ with multiplicity $k$, Eq. \eqref{spectralradius} implies $\nu_1 = \rho(\Gamma)$ and $m_1=k$.\\ Suppose that $\Gamma$ is not bipartite. From Eq. \eqref{contra} we have:
\begin{align*}
0&\leq \frac{m_p \nu_p^{2h}}{k \rho(\Gamma)^{2h}}\cdot \frac{k \rho(\Gamma)^{2h}}{(Tr(|A|_\pi)^{2h})} = \frac{m_p \nu_p^{2h}}{(Tr(|A|_\pi)^{2h})} \leq \frac{Tr\left((A_\pi)^{2h}\right)-k \rho(\Gamma)^{2h}}{(Tr(|A|_\pi)^{2h})}\\&\leq \frac{Tr\left((|A|_\pi)^{2h}\right)-k \rho(\Gamma)^{2h}}{(Tr(|A|_\pi)^{2h})} = 1 -  \frac{k \rho(\Gamma)^{2h}}{(Tr(|A|_\pi)^{2h})}.
\end{align*}
Eqs. \eqref{exp} and \eqref{tracciaproof} imply  $\frac{k \rho(\Gamma)^{2h}}{(Tr(|A|_\pi)^{2h})}\to 1$, and therefore it must be $\frac{m_p \nu_p^{2h}}{k \rho(\Gamma)^{2h}}\to 0$.
This is possible only if $|\nu_p|<\rho(\Gamma)$, and then $|\nu_i |<\rho(\Gamma)$ for $i=2,\ldots,p$. Then it follows, by Eq. \eqref{tracciaproof} and Eq. \eqref{exp}:
$$
\frac{Tr(A_\pi^{h})}{Tr(|A|_\pi^{h})}=\frac{Tr(A_\pi^{h})}{k \rho(\Gamma)^{h}}\cdot \frac{k \rho(\Gamma)^{h}}{Tr(|A|_\pi^{h})}= \frac{k \rho(\Gamma)^{h}+ \sum_{i=2}^p m_i \nu_i^h}{k \rho(\Gamma)^{h}} \cdot \frac{k \rho(\Gamma)^{h}}{k |\mathcal C^h|}   \to 1.
$$
In the bipartite case, one can prove that $\nu_p=-\nu_1$ and $m_1 = m_p=k$, from which we derive the claim, by using an analogous argument.\\
(iv)$\implies$(i)\\
Let us put
\begin{equation*}
\delta:= \min \{ k-\Re(\chi_\pi(g)): g\neq1_G\};
\end{equation*}
notice that, since $\pi$ is faithful, by Lemma \ref{reale}, we have $k-\Re(\chi_\pi(g))>0$ and so, being $G$ finite, we have $\delta>0$. Then, by  Eq. \eqref{miserve}, we have
\begin{equation}\label{mag}
Tr\left((|A|_\pi)^h\right)-Tr\left((A_\pi)^h\right)\geq |{\mathcal C}_u^h|\delta.
\end{equation}
Since $Tr\left((|A|_\pi)^h\right)=k |{\mathcal C}^h|$, Eq. \eqref{mag} gives, in the non-bipartite case:
$$\frac{|{\mathcal C}_u^h|}{|{\mathcal C}^h|}\frac{\delta}{k}\leq\frac{Tr\left((|A|_\pi)^h\right)-Tr\left((A_\pi)^h\right)}{Tr\left((|A|_\pi)^h\right)}=\left(1-\frac{Tr\left((A_\pi)^h\right)}{Tr\left((|A|_\pi)^h\right)}\right)\to 0,$$
which implies $\frac{|{\mathcal C}_u^h|}{|{\mathcal C}^h|}\to 0$. An analogous argument for the bipartite case gives  $\frac{|{\mathcal C}_u^{2h}|}{|{\mathcal C}^{2h}|}\to 0$. By Theorem \ref{thtraccia} the balance of $(\Gamma,\Psi)$ is proved.
\end{proof}

\begin{example}\label{exquaternionibello}\rm
Consider the balanced gain graph in Example \ref{esempio1quaternioni}, and the $2$-dimensional unitary faithful representation $\pi$ of $Q_8$ given in Table \ref{tableQ8}.
\begin{table}
\begin{tabular}{|c|c|c|c|}
\hline
$\pm 1$ & $\pm i$ & $\pm j$ & $\pm k$ \\
\hline
$\pm \left(
   \begin{array}{cc}
     1 & 0 \\
     0 & 1  \\
   \end{array}
 \right)$  &  $\pm\left(
   \begin{array}{cc}
     0 & -1 \\
     1 & 0  \\
   \end{array}
 \right)$ & $\pm \left(
   \begin{array}{cc}
     0 & i \\
     i & 0  \\
   \end{array}
 \right)$ & $\left(
   \begin{array}{cc}
     -i & 0 \\
     0 & i  \\
   \end{array}
 \right)$  \\
\hline
\end{tabular}
\smallskip
\caption{The representation $\pi$ of $Q_8$ of Example \ref{exquaternionibello}.}\label{tableQ8}
\end{table}
A direct computation gives:
$$
A^+ = \left(
        \begin{array}{cccc}
          0 & 1 & 1 & 1 \\
          1 & 0 & 1 & 0 \\
          1 & 1 & 0 & 1 \\
          1 & 0 & 1 & 0 \\
        \end{array}
      \right)   \qquad \sigma(A^+) = \left\{-1,0 , \frac{1\pm \sqrt{17}}{2}\right\}
$$
$$
A_\pi = \left(
        \begin{array}{c|c|c|c}
          O_{2,2} & \pi(-k) & \pi(i) & \pi(k) \\  \hline
          \pi(k) & O_{2,2} & \pi(j) & O_{2,2} \\  \hline
          \pi(-i) & \pi(-j) & O_{2,2} & \pi(j) \\  \hline
          \pi(-k) & O_{2,2} & \pi(-j) & O_{2,2} \\
        \end{array}
      \right)   \quad \sigma(A_\pi) = \left\{(-1)^{2},0^{2}, \left(\frac{1\pm \sqrt{17}}{2}\right)^{2}\right\},
$$
where the exponent  is the multiplicity of the corresponding eigenvalue.
\end{example}

\begin{remark}\rm
Theorem \ref{teo} can be used every time we have a unitary faithful representation of $G$ of finite degree. This is always possible for finite groups, e.g., by using the left regular representation.
Among compact groups, the existence of  a (strongly continuous) finite dimensional faithful representation characterizes  Lie groups (see \cite[Theorem~5.13]{folland}).
Notice that if $G$ is a subgroup of $K$,  a gain graph $(\Gamma,\Psi)$ on the group $G$ can be seen also as a $K$-gain graph, and the balance in $G$ is equivalent to the balance in $K$.
This implies that we can always restrict to the case in which $G$ is generated by its finite subset  $\{ \Psi(u,v):(u,v) \in E_\Gamma\}$. Moreover, even if a unitary representation $\pi$ of finite degree is not faithful, we can apply Theorem \ref{teo} to the associated gain graph on the group $G/\ker (\pi)$.
\end{remark}

\subsection{The case $G=\mathbb T$ with the canonical representation}
Consider the complex unit group $\mathbb T=\{z\in \mathbb C: |z|=1\}$ of the unitary elements of $\mathbb C$ endowed with the natural product. On this particular group a gain graph has a natural adjacency matrix in $M_n(\mathbb C)$, which was already considered in \cite{reff1}.
Therein it is also shown that if a $\mathbb T$-gain graph is balanced, then such adjacency matrix is cospectral with the adjacency matrix of the underlying graph.
Zaslavsky  in \cite{zasbib} asked for the converse implication, that is the generalization on $\mathbb T$-gain graphs of the main result in \cite{acharya} on signed graphs. In \cite[Theorem~4.6]{adun} this generalization is given. However, we present it with an independent proof as a corollary of Theorem \ref{teo}.
\begin{corollary}\label{toro}
A gain graph on $\mathbb T$ is balanced if and only if its adjacency matrix is cospectral with the adjacency matrix of its underlying graph.
\end{corollary}
\begin{proof}
There is a canonical isomorphism  $\pi_{id} \colon \mathbb T\to U_1(\mathbb C)$, associating  with $z$  the $1\times 1$ matrix $(z)$: in particular, the map $\pi_{id}$ is a faithful unitary representation of degree $1$.
For a $\mathbb T$-gain graph $(\Gamma,\Psi)$, with adjacency matrix $A\in M_n(\mathbb C \mathbb T)$, the matrix $A_{\pi_{id}}$ corresponds to the classical adjacency matrix for the complex unit gain graph (in the sense of \cite{reff1}). The claim follows by using Theorem \ref{teo}.
\end{proof}

The finite cyclic group $G=\{1,\xi,\xi^2, \ldots, \xi^{k-1} \}$ naturally embeds in $\mathbb T$ via $\xi^\ell \mapsto e^\frac{2\pi i\ell}{k}$. Then we can apply Corollary \ref{toro} about gain graphs on $\mathbb{T}$. Notice that, when $G$ is the multiplicative group of a finite field, Corollary \ref{toro} can be compared with the main result in \cite{shahulgermina}, where the entries of the adjacency matrix and the coefficients of its characteristic polynomial live in the associated finite field.
When $k=2$, as already mentioned, Corollary \ref{toro} implies the main result in \cite{acharya} for signed graphs.
\\Moreover, when we consider a connected gain graph $(\Gamma,\Psi)$ on a finite subgroup of $\mathbb T$, also the two last equivalences in Theorem \ref{teo} hold.
Let us denote by $\lambda_1(\Gamma,\Psi)$ the largest eigenvalue of $A_{\pi_{id}}$ (sometimes addressed as the \emph{index} of $(\Gamma,\Psi)$)
and with $\lambda_1(\Gamma)$ the largest eigenvalue of the underlying graph $\Gamma$.

\begin{corollary}\label{primoautovalore}
A connected gain graph $(\Gamma,\Psi)$ on a finite subgroup of $\mathbb T$ is balanced if and only if $\lambda_1(\Gamma,\Psi)=\lambda_1(\Gamma)$.
\end{corollary}
\begin{proof}
If $(\Gamma,\Psi)$ is balanced, then by Corollary \ref{toro} the adjacency matrix of $(\Gamma,\Psi)$ and the adjacency matrix of the underlying graph $\Gamma$ are cospectral, in particular $\lambda_1(\Gamma,\Psi)=\lambda_1(\Gamma)$. Vice versa, since $\lambda_1(\Gamma)=\rho(\Gamma)$, if $\lambda_1(\Gamma,\Psi)=\lambda_1(\Gamma)$ then $\rho(\Gamma)$ is in the spectrum of
$A_{\pi_{id}}$. Since the degree of $\pi_{id}$ is $1$, by virtue of Theorem \ref{teo} the gain graph $(\Gamma,\Psi)$ is balanced.
\end{proof}
Notice that, when  $\Gamma$ is regular, Corollary \ref{primoautovalore} is equivalent to the fact that a connected graph $(\Gamma,\Psi)$ is balanced if and only if the Laplacian matrix is not invertible. As observed in \cite[Corollary~3.4]{reff1}, this can be proved, even for non-regular $\mathbb T$-gain graph,  from \cite[Theorem~2.1]{zas4}. In the case $G=\{1,-1\}$ Corollary \ref{primoautovalore} implies that for  a connected graph $\Gamma$, the signatures $\Psi$ inducing balanced signed graphs are exactly those with maximal index. For more sophisticated bounds in this particular case see \cite{stanic2, stanic1}.

\subsection{The case of finite groups}
If $G$ is a finite group, we have faithful unitary representations of finite degree (e.g., the left regular representation $\lambda_G$). Moreover, the complete reducibility of representations (Eq. \eqref{deco}) ensures that it is enough to check balance only at irreducible representations.
\begin{corollary}\label{bfiniti}
Let $G$ be a finite group and  let $(\Gamma,\Psi)$ be a $G$-gain graph, with adjacency matrix $A\in M_n(\mathbb C G)$.  Let $A^+\in M_n(\mathbb C)$  be the adjacency matrix of the underlying graph.
The following are equivalent:
\begin{enumerate}
\item[(i)] $(\Gamma,\Psi)$ is balanced;
\item[(ii)] $\sigma(A_\pi)= \deg(\pi) \sigma( A^+)$ for every unitary irreducible representation $\pi$.
\end{enumerate}
Morover, if $\Gamma$ is connected, also the following is equivalent to (i) and (ii):
\begin{enumerate}
\item[(iii)]   $\lambda_1(A_\pi)=\lambda_1(A^+)$ with multiplicity $\deg(\pi)$, for every unitary irreducible representation $\pi$.
\end{enumerate}
\end{corollary}
\begin{proof}${}$\\
(i)$\iff$(ii)\\
If $\sigma(A_\pi)= \deg(\pi) \sigma( A^+)$, by virtue of Theorem \ref{teo}, the graph $(\Gamma,\Psi)$ is balanced for gain in the group $G/\ker (\pi)$. This implies that, for every  $W\in \mathcal C^h$, we have
$$
\Psi(W)\in \bigcap_{\pi \mbox{ {\tiny irreducible}}} \ker (\pi)=\{1_G\},
$$
that is equivalent to the balance of $(\Gamma,\Psi)$ in $G$. Vice versa, if $(\Gamma,\Psi)$ is balanced for $G$, it is balanced for every quotient group $G/\ker (\pi)$ and, combining again with Theorem \ref{teo}, the claim is proved.
\\By an  analogous  argument, using Theorem \ref{teo}, the equivalence (i) $\iff$ (iii) follows.
\end{proof}

\begin{remark}\label{abelian}\rm
If $(\Gamma,\Psi)$ is a gain graph on a finite Abelian group $G$, each irreducible representation $\pi$  has degree $1$ and, by taking $\pi$ also unitary, we have $A_\pi\in M_n(\mathbb T)$. In particular, $A_\pi$ represents the adjacency matrix of a $\mathbb T$-gain graph $(\Gamma,\pi\circ \Psi)$, obtained from  $(\Gamma,\Psi)$ by replacing each gain $\Psi(u,v)\in G$ with its character $\chi_\pi(\Psi(u,v))$. Under these assumptions, the graph $(\Gamma,\Psi)$ is balanced if and only if $A_\pi$ is cospectral with $A^+$ for each irreducible representation of $G$. Moreover if $\Gamma$ is connected, the $G$-gain graph $(\Gamma,\Psi)$ is balanced if and only if $\lambda_1(\Gamma,\pi\circ \Psi)=\lambda_1(\Gamma)$ for each  irreducible representation of $G$.
\end{remark}

\section{Adjacency matrix of the cover graph}\label{cover}
A natural and useful tool in the investigation of a $G$-gain graph $(\Gamma,\Psi)$ is the  \emph{cover graph $L(\Gamma,\Psi)$}, which is a non-gain graph on $n|G|$ vertices (with $|V_\Gamma|=n$) that, together with a labelling of the vertices, contains all information about $(\Gamma,\Psi)$.
Actually in topological graph theory, a \emph{permutation voltage  assignment}, that is the analogue of the gain function in our context, is defined exactly in order to characterize a graph covering \cite{gross}; for a gain graph point of view, where the switching equivalence plays a fundamental role, see \cite[Section~9]{zasign}.
In the recent literature the cover graph of a $G$-gain graph is defined, for example, in \cite{dalfo,zasglos}: in both cases the multiplication of the gain is on the right. However, we are going to define the cover graph via left multiplication, consistently with Section \ref{circulant} and with the choice of using the left regular representation $\lambda_G$. Anyway, all the results of this section can be adapted for the classical definition by using the \emph{right regular representation}.

\begin{definition}\label{coverdef}
The \emph{cover graph $L(\Gamma,\Psi)$} of the gain graph $(\Gamma,\Psi)$ is the non-oriented graph with vertex set given by $V_\Gamma\times G$, such that, for every $u,v\in V_\Gamma$ and $g,h\in G$ we have:
$$
(u,g)\sim (v,h) \iff (u,v)\in E_\Gamma \mbox{ and }  h=\Psi(u,v)^{-1} g.
$$
\end{definition}

An order on $G=\{g_1,g_2,\ldots, g_k\}$ induces a lexicographic order on the vertex set
$$
V_\Gamma\times G= \{(v_1,g_1),(v_1,g_2),\ldots, (v_n,g_{k-1} ), (v_n,g_k)\}
$$
and defines an adjacency matrix of $L(\Gamma,\Psi)$. The following lemma holds.

\begin{lemma}\label{coverad}
The adjacency matrix of $L(\Gamma,\Psi)$ is $A_{\lambda_G}$, i.e., the Fourier transform of the adjacency matrix of $(\Gamma,\Psi)$ at the left regular representation $\lambda_G$.
\end{lemma}
\begin{proof}
The adjacency matrix $A\in \mathbb M_n(\mathbb C G)$  of $(\Gamma,\Psi)$ can be regarded as an element of $C_{f}(G, M_n(\mathbb C))$ (see Proposition \ref{isomorfismo}).
Explicitly, for every $g\in G$, the matrix $A(g)\in M_n(\mathbb C)$ is such that
$$
A(g)_{i,j}=\begin{cases} 1 &\mbox{ if } (v_i, v_j)\in E_\Gamma \mbox{ and } \Psi(v_i,v_j)=g\\
0 & \mbox{ otherwise.}
\end{cases}
$$
Then, recalling that $A_{\lambda_G}= \hat{A}(\lambda_G)$, by  Definition \ref{fouriermatrici}  we have, for each $i,j=1,\ldots, n$ and $r,p=1,\ldots, k$,
\begin{eqnarray*}
 \hat{A}(\lambda_G)_{k(i-1)+r,k(j-1)+p}&=&\sum_{g\in G} (A(g)\otimes \lambda_G(g))_{k(i-1)+r,k(j-1)+p}
= \sum_{g\in G} A(g)_{i,j} \lambda_G(g)_{r,p}\\&=& \begin{cases} 1 &\mbox{if }  (v_i, v_j)\in E_\Gamma \mbox{ and }  g_p=\Psi(v_i,v_j)^{-1} g_r\\
0 &\mbox{otherwise}
\end{cases}
\end{eqnarray*}
that coincides exactly with the adjacency rules in $L(\Gamma,\Psi)$ of Definition \ref{coverdef}.
\end{proof}
Lemma \ref{coverad} establishes a bridge from our formalism to that used in \cite{dalfo}.  As a first application we can prove,  only by using  linear algebra, that  a switching equivalence between gain graphs  lifts to an isomorphism between their covers (see \cite[Theorem~9.1]{zasign}).
\begin{proposition}\label{lift}
If the gain graphs $(\Gamma,\Psi_1)$ and $(\Gamma,\Psi_2)$ are switching equivalent, then their cover graphs $L(\Gamma,\Psi_1)$ and $L(\Gamma,\Psi_2)$ are isomorphic.
\end{proposition}
\begin{proof}
Since $(\Gamma,\Psi_1)\sim(\Gamma,\Psi_2)$, by virtue of Theorem \ref{swe}, there exists a diagonal $F\in M_n(\mathbb C G)$, with $F_{i,i}\in G$ for each $i$, such that $F^*AF=B$, where $A$ and $B$ are the adjacency matrix of $(\Gamma,\Psi_1)$ and $(\Gamma,\Psi_2)$, respectively. By using the fact that the Fourier transform is multiplicative and commutes with the involution $^*$ (see Proposition \ref{productfou} and take in account  that $\lambda_G$ is unitary), we have
$$
B_{\lambda_G}=\hat{B}(\lambda_G)=\widehat{(F^*AF)}(\lambda_G)= \widehat{F^*}(\lambda_G) \hat{A}(\lambda_G)\hat{F}(\lambda_G)= \hat{F}(\lambda_G)^* A_{\lambda_G} \hat{F}(\lambda_G).
$$
By Lemma \ref{coverad} the adjacency matrices of $L(\Gamma,\Psi_1)$ and $L(\Gamma,\Psi_2)$ are $A_{\lambda_G}$ and $B_{\lambda_G}$, respectively; the thesis follows by noticing that $\hat{F}(\lambda_G)$ is a permutation matrix.
\end{proof}

\begin{remark}\rm
The proof of Proposition \ref{lift} holds, more generally, if  $(\Gamma_1,\Psi_1)$ and $(\Gamma_2,\Psi_2)$ are \emph{switching isomorphic}  (see \cite{zasglos}). In this general case an analogue of Theorem \ref{swe} holds and the adjacency matrices of the gain graphs are conjugated by the product of a diagonal matrix and a \lq\lq permutation matrix\rq\rq $P\in M_n(\mathbb C G)$. The vice versa of Proposition \ref{lift} is false, as we show in the next example.
\end{remark}

\begin{example}\label{examplecontro}\rm
In Fig. \ref{figurecontro} we have represented two gain graphs on $3$ vertices on the cyclic group of $5$ elements $C_5 =\{\xi^j, j=0,1,2,3,4 : \xi=e^{\frac{2\pi i}{5}}\}$. One can check that, although they are not switching isomorphic, their cover graphs are both isomorphic to a cyclic graph on $15$ vertices.
\begin{figure}[h]
\begin{center}
\psfrag{v1}{$v_1$}\psfrag{v2}{$v_2$}\psfrag{v3}{$v_3$}
\psfrag{u1}{$u_1$}\psfrag{u2}{$u_2$}\psfrag{u3}{$u_3$}
\psfrag{id}{$1_{C_5}$}\psfrag{xi}{$\xi$}\psfrag{xi2}{$\xi^2$}
\includegraphics[width=0.4\textwidth]{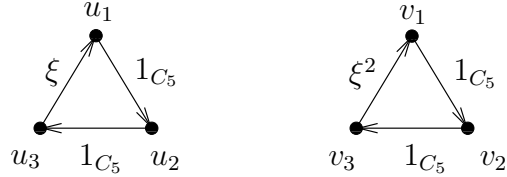}
\end{center}\caption{The gain graphs of Example \ref{examplecontro}.}  \label{figurecontro}
\end{figure}
\end{example}

The relationship between the cover graph and the represented adjacency matrix with respect to a faithful representation allows us to apply the results about balance from Section \ref{sectionbalancenovembre} also in the cover graph setting.

\begin{theorem}\label{coverb}
Let $(\Gamma,\Psi)$ be a connected gain graph  on a group $G$ of order $k$, and $L(\Gamma,\Psi)$ its cover graph.
 The following are equivalent:
\begin{enumerate}
\item[(i)]  $(\Gamma,\Psi)$ is balanced;
\item[(ii)] $\sigma(L(\Gamma,\Psi))=k \sigma(\Gamma)$;
\item[(iii)] $\lambda_1(L(\Gamma,\Psi))=\lambda_1(\Gamma)$ with multiplicity $k$;
\item[(iv)] $L(\Gamma,\Psi)= k\, \Gamma$.
\end{enumerate}
\end{theorem}
\begin{proof} The implication (iv)$\implies$ (ii) is obvious.\\
(i)$\iff$(ii)$\iff$(iii)\\
By  Lemma \ref{coverad}, the adjacency matrix of $L(\Gamma,\Psi)$ is $A_{\lambda_G}\in M_{nk}(\mathbb C)$. The left regular representation $\lambda_G$ is unitary, faithful, of degree $k$. The equivalence follows from Theorem~\ref{teo}.\\
(i)$\implies$(iv)\\
 If $(\Gamma,\Psi)$ is balanced, it is switching equivalent to $(\Gamma,\bold{1}_G)$. As a consequence of Proposition~\ref{lift}, the graph $L(\Gamma,\Psi)$ is isomorphic to $L(\Gamma,\bold{1}_G)$ that, from Definition \ref{coverdef}, clearly consists of $k$ disjoint copies of $\Gamma$.
\end{proof}

Even when the graph is not balanced, our formalism says something about the spectrum of the cover graph. The following corollary is equivalent to the main result in \cite{dalfo,mizuno} in the case of a symmetric voltage assignment. Moreover, it is a generalization of \cite[Theorem~3.2]{analgebraic} for cyclic groups and of \cite{mobius} for signed graphs. It can be alternatively proved by using the main result in \cite{g-circ} and noticing that the adjacency matrix of $L(\Gamma,\Psi)$, with a suitable ordering of the vertices, is group-block circulant (see Section \ref{circulant}).

\begin{corollary}\label{teofinale}
Let $(\Gamma,\Psi)$ be a gain graph on a finite group $G$ and let $A\in M_n(\mathbb C G)$ be its adjacency matrix. Let $L(\Gamma,\Psi)$ be the cover graph of $(\Gamma,\Psi)$ and $\sigma(L(\Gamma,\Psi))$ its adjacency spectrum. We have
$$
\sigma(L(\Gamma,\Psi))=\bigsqcup_{\tiny \pi \mbox{ irreducible}} \deg(\pi)\, \sigma( A_\pi).
$$
\end{corollary}
\begin{proof}
By Lemma \ref{coverad}, the adjacency matrix of $L(\Gamma,\Psi)$ is $A_{\lambda_G}\in M_{n|G|}(\mathbb{C})$. By using the decomposition of $\lambda_G$ into the direct sum of all irreducible representations of $G$ of Eq. \eqref{regdec} and combining with Proposition \ref{productfou}, we deduce that $A_{\lambda_G}$ is similar to the matrix $\bigoplus_{\tiny \pi \mbox{ irreducible}}A_\pi^{ \oplus \deg(\pi)}$. The claim follows.
\end{proof}

Observe that among the irreducible representations of a group $G$, we always have the trivial representation $\pi_0 \colon G\to U_1(\mathbb C)$, $g\mapsto 1$ for all $g\in G$, and $A_{\pi_0}$ is exactly the adjacency matrix of the underlying graph. This implies that the spectrum of $L(\Gamma,\Psi)$, being $(\Gamma,\Psi)$ balanced or not, always contains the spectrum of the underlying graph, and in particular, $ \lambda_1(L(\Gamma,\Psi))=\lambda_1(\Gamma)$ (by Eq. \eqref{spectralradius} in fact, we also have $\lambda_1(A_\pi)\leq \lambda_1(\Gamma)$ for every unitary representation $\pi$).

Notice that, for a group of order $k$, we have a natural embedding $\phi \colon G \hookrightarrow Sym(k)$ given by the left multiplication on $G$ itself. Then  a $G$-gain graph $(\Gamma, \Psi)$ can be seen as a $Sym(k)$-gain graph $(\Gamma , \phi\circ \Psi)$ with adjacency matrix $A\in M_n(\mathbb C Sym(k))$. In particular, the adjacency matrix of the cover graph $L(\Gamma,\Psi)$ is exactly  $A_{\varrho_{k}}$, where $\varrho_{k}$ is the permutation representation of $Sym(k)$.
More generally, if $(\Gamma, \Psi)$ is a gain graph on a subgroup of $Sym(m)$, the matrix  $A_{\varrho_{m}}$ is still the adjacency matrix of a graph that covers, in a precise sense, the gain graph $(\Gamma,\Psi)$ (see \cite{gross} and, for recent developments, \cite{arbitro}). Notice that anyway, by Proposition \ref{productfou} and the inclusion of $\pi_0$ in $\varrho_m$, the spectrum of $A_{\varrho_{m}}$ contains the spectrum of the underlying graph.

Thanks to our formalism, we can push forward this approach and generalize it  to $\mathbb T$-gain ``cover'' graphs.
Suppose that we have a $\mathbb T$-gain graph with $m$ vertices, with complex adjacency matrix $M\in M_m(\mathbb C)$.
Now suppose that:
\begin{itemize}
\item $\pi$ is a unitary representation of a group $G$ such that $\pi\sim \pi_1\oplus\cdots\oplus \pi_l$, with $\pi_i$ irreducible and $\deg(\pi)=k$;
\item  $(\Gamma,\Psi)$ is a  $G$-gain graph on $n$ vertices and adjacency matrix $A\in M_n(\mathbb C G)$;
\item $m=nk$ and  $M=A_{\pi} = \widehat{A}(\pi)$.
\end{itemize}
Then we have:
$ \sigma(M)=\sigma(A_{\pi_1}) \cup \ldots \cup\sigma(A_{\pi_l}).$

The following simple example suggests the existence of new possible applications of this point of view.
\begin{example}\label{exampleklein}\rm
Let us consider the Klein group $V=<a,b : a^2=b^2 = (ab)^2=1_V>$, which is the commutative group of $4$ elements $V=\{1_V, a,b,c\}$, where we put $c=ab$. Consider now the gain graph $(\Gamma,\Psi)$ on $V$ in Fig. \ref{figuradibase} (we do not need to orient edges, since all the nontrivial elements of $V$ are involutions)
\begin{figure}[h]
\begin{center}
\psfrag{1}{$1$}\psfrag{2}{$2$}\psfrag{3}{$3$} \psfrag{4}{$4$}
\psfrag{id}{$1_V$}\psfrag{a}{$a$}\psfrag{b}{$b$}\psfrag{c}{$c$}\psfrag{GG}{$\Gamma$}
\includegraphics[width=0.2\textwidth]{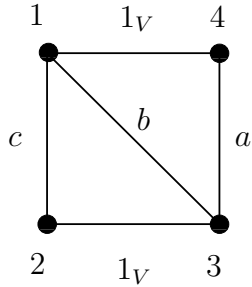}
\end{center}\caption{The gain graph $(\Gamma,\Psi)$ of Example \ref{exampleklein}.}  \label{figuradibase}
\end{figure}
which has adjacency matrix $A= \left(
                              \begin{array}{cccc}
                                0 & c & b & 1_V \\
                                c & 0 & 1_V & 0 \\
                                b & 1_V & 0 & a \\
                                1_V & 0 & a & 0 \\
                              \end{array}
                            \right)
$.
Now consider the $2$-dimensional unitary representation $\pi$ of $V$ given in Table \ref{tableklein}.
\begin{table}
\begin{tabular}{|c|c|c|c|}
\hline
$1_V$ & $a$ & $b$ & $c$ \\
\hline
$\left(
   \begin{array}{cc}
     1 & 0 \\
     0 & 1\\
   \end{array}
 \right)$  &  $\left(
   \begin{array}{cc}
     -1 & 0 \\
     0 & -1 \\
   \end{array}
 \right)$  & $\left(
   \begin{array}{cc}
     0 & 1 \\
     1 & 0 \\
   \end{array}
 \right)$  & $\left(
   \begin{array}{cc}
     0 & -1 \\
     -1 & 0 \\
   \end{array}
 \right)$ \\
\hline
\end{tabular}
\smallskip
\caption{The representation $\pi$ of $V$ in Example \ref{exampleklein}.} \label{tableklein}
\end{table}
By using the orthogonality relations of characters one can check that $\pi \sim \pi_1 \oplus \pi_2$, where
$$
\pi_1(1_V) = 1, \ \pi_1(a) = -1, \ \pi_1(b) = 1, \ \pi_1(c) = -1
$$
and
$$
\pi_2(1_V) = 1, \ \pi_2(a) = -1, \ \pi_2(b) = -1, \ \pi_2(c) = 1.
$$
Now the matrices $A_{\pi_1}$, $A_{\pi_2}$, and $A_\pi$, which are all symmetric matrices with entries in $\{-1,0,1\}$, can be regarded as the adjacency matrices of the signed graphs $\Gamma_{\pi_1}$, $\Gamma_{\pi_2}$, and $\Gamma_{\pi}$, respectively, represented in Fig. \ref{itre}. Continuous lines denote positive edges and dashed lines denote negative edges.
\begin{figure}[h]
\begin{center}
\psfrag{1}{$1$}\psfrag{2}{$2$}\psfrag{3}{$3$} \psfrag{4}{$4$}\psfrag{5}{$5$}\psfrag{6}{$6$}\psfrag{7}{$7$} \psfrag{8}{$8$}\psfrag{G1}{$\Gamma_{\pi_1}$}\psfrag{G2}{$\Gamma_{\pi_2}$}\psfrag{Gi}{$\Gamma_\pi$}
\includegraphics[width=0.5\textwidth]{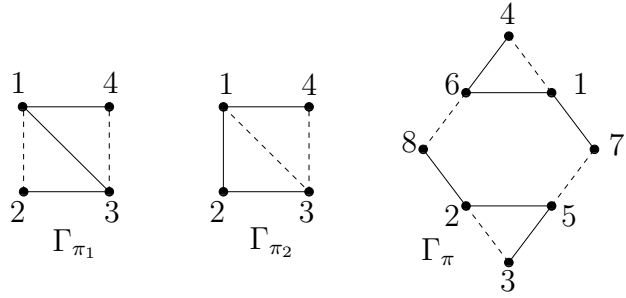}
\end{center}\caption{The signed graphs $\Gamma_{\pi_1}$, $\Gamma_{\pi_2}$, and $\Gamma_{\pi}$ of Example \ref{exampleklein}.}  \label{itre}
\end{figure}
More precisely:
\begin{align*}
A_{\pi_1}= \left(
                                        \begin{array}{cccc}
                                          0 & -1 & 1 & 1 \\
                                          -1 & 0 & 1 & 0 \\
                                          1 & 1 & 0 & -1 \\
                                           1 & 0 & -1 & 0 \\
                                           \end{array}
                                      \right)
\qquad
A_{\pi_2}=\left(
                                        \begin{array}{cccc}
                                          0 & 1 & -1 & 1 \\
                                          1 & 0 & 1 & 0 \\
                                          -1 & 1 & 0 & -1 \\
                                           1 & 0 & -1 & 0 \\
                                           \end{array}
                                      \right)
\end{align*}
and
$$
A_\pi = \left(
                              \begin{array}{c|c|c|c}
                                O_{2,2} & \pi(c) & \pi(b) & \pi(1_V) \\ \hline
                                \pi(c) & O_{2,2} & \pi(1_V) & O_{2,2} \\ \hline
                                \pi(b) & \pi(1_V) & O_{2,2} & \pi(a) \\ \hline
                                \pi(1_V) & O_{2,2} & \pi(a) & O_{2,2} \\
                              \end{array}
                            \right) = \left(
                                        \begin{array}{cc|cc|cc|cc}
                                          0 & 0 & 0 & -1 & 0 & 1 & 1 & 0 \\
                                          0 & 0 & -1 & 0 & 1 & 0 & 0 & 1 \\ \hline
                                          0 & -1 & 0 & 0 & 1 & 0 & 0 & 0 \\
                                           -1 & 0 & 0 & 0 & 0 & 1 & 0 & 0 \\ \hline
                                          0 & 1 & 1 & 0 & 0 & 0 & -1 & 0 \\
                                          1 & 0 & 0 & 1 & 0 & 0 & 0 & -1 \\  \hline
                                          1 & 0 & 0 & 0 & -1 & 0 & 0 & 0 \\
                                          0 & 1 & 0 & 0 & 0 & -1 & 0 & 0 \\
                                        \end{array}
                                      \right).
$$
An explicit spectral computation gives $\sigma(A_{\Gamma_\pi}) = \sigma(A_{\Gamma_{\pi_1}})\cup \sigma(A_{\Gamma_{\pi_2}})$, with
$$
\sigma(A_{\Gamma_{\pi_1}}) = \left\{0,1, \frac{-1\pm \sqrt{17}}{2}\right\}, \ \sigma(A_{\Gamma_{\pi_2}}) = \{\pm 1, \pm 2\}.
$$
Observe that, with respect to the notation in the construction preceding Example \ref{exampleklein}, we have $G=V$; $k=l=2$; $n=4$; $m=8$.
\end{example}

\section*{Acknowledgements}
We want to express our deepest gratitude to Francesco Belardo for introducing us to the beautiful theory of gain graphs and for useful and stimulating discussions. The second author thanks the Austrian Science Fund project FWF P29355-N35.


\end{document}